\documentclass[10pt]{article}

\usepackage{amsmath,amsfonts,amsthm,amscd,amssymb,graphicx}
\numberwithin{equation}{section}

\usepackage{fullpage}
\usepackage{hyperref}

\usepackage{color}

\newtheorem{theorem}{Theorem}[section]

\newtheorem{lemma}[theorem]{Lemma}

\newtheorem{proposition}[theorem]{Proposition}
\newtheorem{prop}[theorem]{Proposition}

\newtheorem{remark}[theorem]{Remark}
\newtheorem{definition}[theorem]{Definition}

\renewcommand{\S}{S}
\renewcommand{\SS}{S'}

\begin{document}

\title{Nonlinear instability of Vlasov-Maxwell systems in the classical and quasineutral limits}

\author{Daniel Han-Kwan\footnotemark[1] \and Toan T. Nguyen\footnotemark[2]}

\date{\today}

\maketitle

\renewcommand{\thefootnote}{\fnsymbol{footnote}}

\footnotetext[1]{CNRS $\&$ \'Ecole polytechnique, Centre de Math\'ematiques Laurent Schwartz UMR 7640, 91128 Palaiseau Cedex, France. \\Email: daniel.han-kwan@polytechnique.edu}
\footnotetext[2]{Department of Mathematics, Pennsylvania State University, State College, PA 16802, USA. \\Email: nguyen@math.psu.edu}

\begin{abstract}

We study the instability of solutions to the relativistic Vlasov-Maxwell systems in two limiting regimes: the classical limit when the speed of light tends to infinity and the quasineutral limit when the Debye length tends to zero. First, in the classical limit $\varepsilon \to 0$, with $\varepsilon$ being the inverse of the speed of light, we construct a family of solutions that converge initially polynomially fast to an homogeneous solution $\mu$ of Vlasov-Poisson in arbitrarily high Sobolev norms, but become of order one away from $\mu$ in arbitrary negative Sobolev norms within time of order $|\log \varepsilon|$. Second, we deduce the invalidity of the quasineutral limit in $L^2$ in arbitrarily short time.

\end{abstract}

\section{Introduction}
%TODO: modify the relativistic velocity $v$ as in Puel and Saint-Raymond? In this case, $\hat v: = \frac{v}{\sqrt{1+ \varepsilon v^2}}$.  

%\subsection{The quasineutral limit of the Vlasov-Maxwell system}
We study the relativistic Vlasov-Maxwell system
\begin{equation}\label{clVM0} 
\left \{ \begin{aligned}
\partial_t f + \hat{v} \cdot \nabla_x f + (E + \frac{1}{c} \hat{v} \times B)\cdot \nabla_v f & =0,
\\
\frac{1}{c} \partial_t B + \nabla_x \times E = 0, \qquad \nabla_x \cdot E  &= \rho - 1,
\\
- \frac{1}{c} \partial_t E + \nabla_x \times B = \frac{1}{c} j , \qquad \nabla_x \cdot B & =0, \end{aligned}
\right.
\end{equation}
describing the evolution of an electron distribution function $f(t,x,v)$ at time $t\ge 0$, position $x\in \mathbb{T}^3:= \mathbb{R}^3/ \mathbb{Z}^3$, momentum $v \in \mathbb{R}^3$ and relativistic velocity
$$
\hat{v} = \frac{v}{\sqrt{1+  \frac{|v|^2}{c^2}}}. 
$$
In these equations, the parameter $c$ is the speed of light. Here, $\mathbb{T}^3$ is equipped with the normalized Lebesgue measure so that $\mathrm{Leb}(\mathbb{T}^3)=1$. The electric and magnetic fields $E(x,t), B(x,t)$ are three-dimensional vector fields, satisfying the classical Maxwell equations, with sources given by 
$$\rho (t,x)= \int_{\mathbb{R}^3} f \; dv, \quad j (t,x)= \int_{\mathbb{R}^3} \hat{v} f\; dv,$$ 
which denote the usual charge density and current of electrons. The background ions are assumed to be homogeneous with a constant charge density equal to one. 

We focus in the regime where $c \to +\infty$, that is known as the \emph{classical (or non-relativistic) limit of the Vlasov-Maxwell system}. The formal limit is the classical Vlasov-Poisson system:
\begin{equation}\label{VP} 
\left \{ \begin{aligned}
\partial_t f + {v} \cdot \nabla_x f + E \cdot \nabla_v f & =0,
\\
 \nabla_x \times E = 0, \qquad \nabla_x \cdot E  &= \rho - 1,
\end{aligned}
\right.
\end{equation}
where $\rho(t,x) = \int_{\mathbb{R}^3} f \, dv$.
This limit was justified on finite intervals of time in the independent and simultaneous works of Asano-Ukai \cite{AU}, Degond \cite{DEG}, and Schaeffer \cite{SCH}.
 We shall prove in this paper that \emph{in the classical limit, the Vlasov-Maxwell system can develop instabilities in time $\log c$, due to instabilities of the limiting Vlasov-Poisson system}.

 \bigskip

We are also interested in the following non-dimensional Vlasov-Maxwell system:
\begin{equation}\label{VM} 
\left \{ \begin{aligned}
\partial_t f + \hat{v} \cdot \nabla_x f + (E + \alpha \hat{v} \times B)\cdot \nabla_v f & =0,
\\
\alpha \partial_t B + \nabla_x \times E = 0, \qquad \varepsilon^2 \nabla_x \cdot E  &= \rho - 1,
\\
- \alpha \varepsilon^2 \partial_t E + \nabla_x \times B = \alpha j, \qquad \nabla_x \cdot B & =0, \end{aligned}
\right.
\end{equation}
with
$$
\begin{aligned}
&\hat{v} = \frac{v}{\sqrt{1+  \alpha{\varepsilon^2 |v|^2}}}, 
 \quad \rho (t,x)= \int_{\mathbb{R}^3} f \; dv, \quad j (t,x)= \int_{\mathbb{R}^3}\frac{v}{\sqrt{1+  \alpha{\varepsilon^2 |v|^2}}} f\; dv. 
\end{aligned}
$$
In physical units, the parameters $\alpha, \varepsilon$ are given by  
$$ \alpha = \sqrt{\frac{r_0}{\varepsilon_0}}  , \qquad \varepsilon = \sqrt{\frac{\varepsilon_0}{r_0 c^2}},$$
where $r_0$ denotes the classical electron radius, $\varepsilon_0$ is the vacuum dielectric constant, and $c$ is still the speed of light. The parameter $\varepsilon$ corresponds to the classical \emph{Debye length}  of the electrons; see \cite{PSR} for further discussions. In this work, we are interested in the regime where $\alpha \sim 1$ and $\varepsilon \to 0$, a limit in which the charge density of ions and electrons are formally equal. %to each other, %plus possible spatial oscillations at the scale of order $\varepsilon$. 
We shall thus refer to this problem as the \emph{quasineutral limit of the Vlasov-Maxwell system}. 
Note that $\alpha$ is equal to the ratio between $\frac{1}{\varepsilon}$ and $c$, so that this means that we consider that the inverse of the Debye length is of the same order as the speed of light.
 For simplicity, throughout the paper, we set $\alpha=1$.

\bigskip

The quasineutral limit of the Vlasov-Maxwell system has been studied previously by Brenier, Mauser and Puel \cite{BMP} and Puel and Saint-Raymond \cite{PSR} in the case where the initial density distribution converges in some weak sense to a monokinetic distribution (that is, a Dirac delta function in velocity). The convergence to monokinetic distributions can be interpreted from the physical point of view as \emph{vanishing temperature}: therefore, this is sometimes referred to as the \emph{cold electrons limit}.
The work \cite{PSR} furthermore describes the propagation of time oscillating waves, which turn out to be absent in the well-prepared framework considered in \cite{BMP}.

For each fixed $\varepsilon>0$, the global-in-time Cauchy theory for smooth solutions of \eqref{VM} remains an outstanding open problem. However, there are local strong solutions and continuation conditions (\cite{Wol1,Wol2,ASA,GS86, GSC97,GSC98,BGP,KS, LStr}, among others) or global weak solutions (\cite{DPL}). In this paper, we shall construct particular solutions of \eqref{VM} that are sufficiently smooth. This also pertains to system~\eqref{clVM0}, for each fixed $c>0$.

%
%Observe that for any $\varepsilon>0$, one can apply the existing Cauchy theories to get for \eqref{VM} existence of either global weak solutions \cite{DPL} or local strong solutions \cite{Wol1,Wol2,ASA,GS86} (see also \cite{GSC97,GSC98,BGP,KS}).
%

\bigskip

Formally, in the limit $\varepsilon \to 0$, it is straightforward to obtain the expected formal limit of \eqref{VM}, a system we shall call the \emph{kinetic eMHD} system:
\begin{equation}\label{kinetic-eMHD} 
\left \{ \begin{aligned}
\partial_t f^0 + v \cdot \nabla_x f^0 + (E^0 + v \times B^0)\cdot \nabla_v f^0 & =0,
\\
 \partial_t B^0 + \nabla_x \times E^0 = 0, \qquad \rho^0 &= 1,
\\
 \nabla_x \times B^0 = j^0, \qquad \nabla_x \cdot B^0 & =0, \end{aligned}
\right.
\end{equation}
in which 
$$ \rho^0 (t,x)= \int_{\mathbb{R}^3} f^0 \; dv, \quad j^0 (t,x)= \int_{\mathbb{R}^3}v f^0\; dv.
$$
By imposing that the distribution function is monokinetic, that is considering the ansatz 
$$f^0(t,x,v) =\rho^0(t,x) \delta_{v=u^0(t,x)},$$ 
where $\delta$ stands for the Dirac measure, it follows that $f^0$ is a solution in the sense of distribution to \eqref{kinetic-eMHD} if and only if $(\rho^0,u^0)$ satisfies the following hydrodynamic equations:
\begin{equation}\label{eMHD} 
\left \{ \begin{aligned}
\partial_t u^0 + \nabla \cdot (u^0 \otimes u^0) =E^0 + u^0 \times B^0,
\\
 \partial_t B^0 + \nabla \times E^0 = 0, \qquad \rho^0 &= 1,
\\
 \nabla \times B^0 = u^0, \qquad \nabla \cdot B^0 & =0. \end{aligned}
\right.
\end{equation}
This system is known in the literature as the \emph{electron Magneto-Hydro-Dynamics} equations (eMHD); see, for instance, \cite{KCI, BMP}. This motivates the choice of the name \emph{kinetic eMHD} for~\eqref{kinetic-eMHD}. The work \cite{BMP}, and then \cite{PSR}, justify the eMHD system from \eqref{VM} in the above monokinetic situation via the so-called modulated energy (or relative entropy) method devised by Brenier in \cite{B}. In this paper, we rather focus on the question of  {\em validity of \eqref{kinetic-eMHD} in the quasineutral limit.}  

\bigskip

The kinetic eMHD system \eqref{kinetic-eMHD} can be studied as follows. It is common to express $E^0$ and $B^0$ in terms of electromagnetic potentials $(\phi,A)$. Precisely, write 
\begin{equation}\label{potentials}
 E^0 = - \nabla \phi^0 - \partial_t A^0, \qquad B^0 = \nabla \times A^0,
 \end{equation}
in which $\phi^0$ is a scalar function and $A^0$ is a divergence-free vector potential. The equations satisfied by the electromagnetic field are then simply reduced to 
the elliptic equation  
\begin{equation} \label{ellip1}
-\Delta A^0 =  j^0,
\end{equation}
together with the constraint $ \rho^0 = 1$.
On the other hand, the potential $\phi^0$ is determined from the Vlasov equation. Indeed, there holds the conservation of charge density and current
$$
\begin{aligned}\partial_t \rho^0 + \nabla \cdot j^0&= 0,
\\
\partial_t j^0 + \nabla \cdot \int f^0 v \otimes v \, dv &= E^0+ j^0 \times B^0.
\end{aligned}$$
From the neutrality condition $\rho^0 = 1$, we conclude that the first moment $j^0$ is divergence-free. Hence, taking the divergence of the conservation of current, we end up with the following elliptic problem for $\phi^0$:
   \begin{equation} \label{ellip2}
 -\Delta \phi^0 = \nabla \cdot \Big( \nabla \cdot \int f^0 v \otimes v \, dv \Big)-  \nabla \cdot (j^0 \times B^0) .\end{equation}
 
The equations \eqref{potentials}-\eqref{ellip2}, together with the Vlasov equation in \eqref{kinetic-eMHD}, form a complete set of equations for solutions of the kinetic eMHD system. 
However, the elliptic problem \eqref{ellip2} reveals a loss of one $x$-derivative in $E^0$, or precisely in the curl part of $E^0$, as compared to $f^0$. One can expect this loss to be the source of some \emph{ill-posedness} for \eqref{kinetic-eMHD}. This feature is a first strong indication of the singularity of the quasineutral limit.
Note finally that this loss can only occur for \eqref{kinetic-eMHD}, while the eMHD equations are well-posed in the classical sense as shown in \cite{BMP}.

%This is the main source of ill-posedness of the limiting system \eqref{kinetic-eMHD}. 

\bigskip

Lately, there have been many works \cite{Gr95,Gr96,Gr99,B,Mas,HKH,HI,HI2,HKR} that were devoted to the study of the quasineutral limit for Vlasov-Poisson systems, 
that is, the framework where there is no magnetic field, $B\equiv 0$. 
In that case, the formal limit is also straightforward to identify and corresponds to a kinetic version of the incompressible Euler equations. The limiting kinetic system also displays a loss of derivatives, through the electric field $E^0 = -\nabla \phi^0$, exactly as in \eqref{ellip2}. In the paper \cite{HKH}, using ideas originating from \cite{Gr99}, it is shown in particular that instabilities such as \emph{two-stream instabilities} (see \cite{pen,GS}) for the Vlasov-Poisson system give rise to instabilities in the quasineutral limit. These instabilities have a destabilizing effect and the formal limit is not true in general, even
on very short intervals of time $[0,T_\varepsilon]$, with $T_\varepsilon \to 0$ as $\varepsilon \to 0$.

\bigskip

One purpose of this work is to extend these ideas to the Vlasov-Maxwell system~\eqref{VM}.
As in the Vlasov-Poisson case  \cite{HKH}, the effect of instabilities of the Vlasov-Maxwell system in the quasineutral limit can  be observed in the high spatial frequency regime. More precisely, we introduce the hyperbolic change of variables: 
\begin{equation}\label{e-hyp} 
(t,x,v)\quad \mapsto\quad  (s,y,v): = \left(\frac{t}{\varepsilon}, \frac{x}{\varepsilon}, v\right)
\end{equation}
and set $E^\varepsilon = \varepsilon E$, $B^\varepsilon = B$, and $f^\varepsilon = f$. The Vlasov-Maxwell system \eqref{VM} then becomes  
\begin{equation}\label{clVM} 
\left \{ \begin{aligned}
\partial_s f^\varepsilon + \hat{v}  \cdot \nabla_y f^\varepsilon + (E^\varepsilon + \varepsilon \hat{v}  \times B^\varepsilon)\cdot \nabla_v f^\varepsilon & =0,
\\
\varepsilon \partial_s B^\varepsilon + \nabla_y \times E^\varepsilon = 0, \qquad \nabla_y \cdot E^\varepsilon  &= \rho^\varepsilon - 1,
\\
- \varepsilon \partial_s E^\varepsilon + \nabla_y \times B^\varepsilon = \varepsilon j^\varepsilon , \qquad \nabla_y \cdot B^\varepsilon & =0, \end{aligned}
\right.
\end{equation}
where 
\begin{equation}\label{def-rhoje}
\begin{aligned}
\hat{v} &= \frac{v}{\sqrt{1+  {\varepsilon^2 |v|^2}}}, \qquad \rho^\varepsilon (t,x)= \int_{\mathbb{R}^3} f^\varepsilon \; dv, 
\\ j^\varepsilon (t,x) &= \int_{\mathbb{R}^3}\frac{v}{\sqrt{1+  {\varepsilon^2 |v|^2}}} f^\varepsilon\; dv. 
\end{aligned}
\end{equation}

We therefore observe that in this \emph{long time} and \emph{spatially high frequency} regime, the quasineutral limit comes down to the study of the classical limit %in a large time $s=\frac t\varepsilon$, 
of Vlasov-Maxwell to Vlasov-Poisson system, that is~\eqref{clVM0} with the speed of light $c = \frac1 \varepsilon$.
We shall thus mainly focus on the study of the classical limit.

\bigskip

Consequently, it appears natural to study the effect of instability of equilibria of the Vlasov-Maxwell system \eqref{clVM} to deduce instability for~\eqref{VM}.
There have many works devoted to the stability problem over the past few years; see, for instance, \cite{G97, G99, GS99, GS00, LS07, LS07b, LS08} and the references therein. It turns out that in the limit $\varepsilon \to 0$, instability for the Vlasov-Poisson system~\eqref{VP} %of~\eqref{clVM} 
is sufficient to get an instability in the classical limit, thanks to a careful use of the iterative scheme developed by Grenier in his study of instability of boundary layers \cite{G}. In other words, we show that having an \emph{approximate growing mode} of the linearized Vlasov-Maxwell equations is enough.
Loosely speaking, we shall prove that if we deal with sequences of initial data converging polynomially (in $\varepsilon$) fast to an unstable equilibrium, then an instability is developed for the full Vlasov-Maxwell dynamics, in times of order $|\log \varepsilon|$.
We shall finally deduce the invalidity of the derivation of \eqref{kinetic-eMHD} in the quasineutral limit of~\eqref{VM}.

\section{Main results}

Our results rely on the existence of a growing mode for the linearized Vlasov-Maxwell system \eqref{VM}, around unstable \emph{homogeneous} equilibria $(f,E,B) \equiv (\mu(v), 0,0)$.  To ensure the non-negativity of the distribution functions we shall construct, we introduce the following condition: 

%TODO: why delta ?

\begin{definition} [$\delta$-condition]
We say that a profile $\mu(v)$ satisfies the $\delta$-condition if $\mu$ is positive and satisfies
\begin{equation} \label{e-nonnegativity}
\sup_{v \in \mathbb{R}^3}  \,  \frac{|\nabla \mu(v)|}{(1+ |v|)\mu(v)} < + \infty.
\end{equation}
\end{definition}

~\\
The terminology is borrowed from \cite{HKH}; in that reference, a further relaxed $\delta'$-condition is also introduced, allowing non-negative equilibria. Our results here should apply to this $\delta'$-condition as well. For the sake of simplicity, we restrict ourselves to the above $\delta$-condition on $\mu$ throughout the paper. 

\bigskip

In what follows, the statement that $\mu(v)$ is spectrally unstable means that there is a growing mode of the form $e^{\lambda t} g(y,v)$, $g \in L^2$, with $\Re \lambda>0$, of the linearized Vlasov-Poisson problem around the homogeneous equilibrium $\mu(v)$, namely: $(\lambda, g)$ is a solution of 
\begin{equation} \label{e-insta}
\lambda g + v\cdot \nabla_y g + \nabla_v \mu (v) \cdot \nabla_y \Delta_y^{-1}  \left(\int_{\mathbb{R}^3} g \, dv \right)  = 0, \qquad \iint g(y,v) \; dydv =0. 
\end{equation}
%In our proof, we exhibit some sufficient conditions ensuring the spectral instability in the above sense. We emphasize that only instability of the Vlasov-Poisson system is needed. 

%\subsection{Penrose instability criterion}
Looking for a solution of the form $g =  e^{ik\cdot y}\hat g(v)$ %growing mode of the form $f = e^{\lambda s} e^{ik\cdot y}\hat f(v)$, of the linearized Vlasov-Poisson system:
%\begin{equation}\label{clVM-lin} 
%\partial_s f  + v \cdot \nabla f  + \nabla_v \mu \cdot \nabla \Delta^{-1} \rho (f) = 0, 
%\end{equation}
and writing $\lambda = - ik \cdot \omega$, for some complex vector $\omega$, with $\Im \omega \cdot k >0$,
we are led to study%. The above equation reduces to 
 $$ ik \cdot ( v - \omega ) \hat g (v) -  \frac{1}{|k|^2} i k \cdot \nabla_v \mu   \, \hat \rho = 0,$$  
in which $\hat \rho = \int \hat g(v)\; dv$. The above identity has a nonzero solution if and only if $(k,\omega)$ is such that the so-called \emph{Penrose instability criterion}
\begin{equation}\label{Penrose}
 \frac{1}{|k|^2} \int \frac{k \cdot \nabla_v \mu}{ k \cdot ( v - \omega ) } \; dv = 1, \quad \Im \omega \cdot k >0,
 \end{equation}
 is satisfied.
We have the following classical lemma (see \cite{pen,Degond}).

\begin{lemma}\label{lem-unmode} Let $\mu$ satisfy the Penrose instability criterion  \eqref{Penrose} at some point $(k_0, \omega_0)$. Then, $\mu$ is spectrally unstable and there exists a growing mode of the form 
\begin{equation}
\label{e-eigen}
f(s,y,v) = e^{\lambda_0 s} e^{ik_0\cdot y}\hat f(v),
\end{equation}
with $\lambda_0 = -i k_0\cdot \omega_0$ and 
\begin{equation}\label{clVM-def} \hat f  (v)=   \frac{1}{|k_0|^2} \frac{ \nabla_v \mu \cdot k_0}{k_0\cdot (v-\omega_0)} . \end{equation}
Reciprocally, if $\lambda$ is an eigenvalue for the linearized operator around $\mu$ then all eigenfunctions are linear combinations of functions of the form~\eqref{e-eigen}~-~\eqref{clVM-def}.
\end{lemma}
%\begin{proof} We note that the Penrose condition \eqref{Penrose} clearly fails when either $|k|\to \infty$ or $|\omega|\to \infty$, since the integral on the left of \eqref{Penrose} tends to zero. This yields the existence of a maximal growing mode among the possible growing modes. The lemma is proved. 
%\end{proof}

%Reciprocally, it turns out that if $\mu$ is spectrally unstable, then all the growing modes are of the form $f= e^{\lambda_0 s} e^{ik_0\cdot y}\hat f(v)$, with $\hat f$ defined in~\eqref{clVM-def}.

In particular, it is known, see for instance \cite{GS}, that if there is a vector $e \in \mathbb{S}^2$ such that the function
$$
\mu_e (r) = \int_{r e + e^\perp} \mu(w) \, dw
$$
admits a local strict minimum at a point $\bar r \in \mathbb{R}$ and is symmetric around $\bar r$, and that  the following inequality holds:
\begin{equation} \label{e-Pen}
\int_{\mathbb{R}} \frac{\mu_e(r)- \mu_e(\bar r)} { |r- \bar r|^2 } \, dr >   4 \pi^2,
\end{equation}
then the Penrose instability criterion is satisfied.

\bigskip

In this paper, we shall only consider equilibria that are 
\begin{itemize}
%\item \emph{radial}, that is to say $\mu \equiv \mu(|v|^2)$;
\item \emph{smooth} (i.e. $C^k$, with $k \gg 1$) and  \emph{decaying sufficiently fast at infinity};
\item \emph{normalized} in the sense that $\int_{\mathbb{R}^3} \mu(v) \, dv =1$
and $
 \int_{\mathbb{R}^3}  \mu(v) \hat{v} \, dv =0,
$
for all $c>0$.
\end{itemize}
%For instance the last condition can be ensured by considering profiles of the form
%$
%\mu(v_i) = \prod_{i=1}^3 \mu_i (v_i^2)
%$. 
%
As a result, any such equilibrium gives a stationary solution for all the systems we study in this paper, namely the relativistic Vlasov-Maxwell system \eqref{VM}, the Vlasov-Poisson system \eqref{VP} and the kinetic eMHD system \eqref{kinetic-eMHD}.

We are now in position to state our first result.

\begin{theorem}[Instability in the classical limit]
\label{t-classical}

Let $\mu(v)$ be a smooth, normalized equilibrium that satisfies the $\delta$-condition and the Penrose instability criterion.
Then, for any $m,\S,\SS,p>0$, there exist a 
family of smooth solutions $(f^{\varepsilon}, E^{\varepsilon}, B^{\varepsilon})_{\varepsilon>0}$ of \eqref{clVM}, with $f^\varepsilon\ge 0$, and a sequence of times $s_\varepsilon = \mathcal{O}(|\log \varepsilon|)$ such that 
\begin{equation}
\| (1+| v|^2)^{\frac m2} ({f^{\varepsilon}}_{\vert_{s=0}}- \mu)\|_{H^\S(\mathbb{T}^3\times \mathbb{R}^3)} \le \varepsilon^p,
\end{equation}
but %there is a sequence of times $s_\varepsilon = \mathcal{O}(|\log \varepsilon|)$ such that
\begin{equation}
\label{e-thm1}
\liminf_{\varepsilon \rightarrow 0}  \left\| f^\varepsilon(s_\varepsilon) - \mu\right\|_{H^{-\SS}(\mathbb{T}^3\times \mathbb{R}^3)} >0,
\end{equation}
\begin{equation}
\label{e-thm2}
 \liminf_{\varepsilon \rightarrow 0}\left\| \rho^\varepsilon(s_\varepsilon) - 1 \right\|_{H^{-\SS}(\mathbb{T}^3)} >0,
\quad
\liminf_{\varepsilon \rightarrow 0}  \left\| j^\varepsilon(s_\varepsilon) \right\|_{H^{-\SS}(\mathbb{T}^3)} >0,
\end{equation}
\begin{equation}
\label{e-thm3}
\liminf_{\varepsilon \rightarrow 0}  \left\| E^\varepsilon(s_\varepsilon) \right\|_{L^2(\mathbb{T}^3)} >0
\end{equation}
in which $\rho^\varepsilon, j^\varepsilon$ are defined as in \eqref{def-rhoje}.
\end{theorem}

\begin{remark}
Consider a smooth one-dimensional double-bump equilibrium (see \cite{GS}), written as $\mu_1(v_1^2)$ and consider two other 1D radial profiles $\mu_2(v_2^2)$ and $\mu_3(v_3^2)$,
normalized so that
%For instance the last condition can be ensured by considering profiles of the form
$$
\mu(v_1,v_2,v_3) := \prod_{i=1}^3 \mu_i (v_i^2)
$$
is such that $\int \mu \, dv =1$ and satisfies the $\delta$-condition. Then we note that  for all $c>0$,
$ \int_{\mathbb{R}^3}  \mu(v) \hat{v} \, dv =0$, so that such equilibria
 satisfy all the required assumptions of the theorem.
\end{remark}

\begin{remark}
This  nonlinear instability result for the  full 3D Vlasov-Maxwell dynamics is concerned with the classical limit $c =1/\varepsilon\to +\infty$; it is stronger in terms of admissible norms than those previously known for fixed values of the speed of light 
$c$, see e.g. \cite{GS00}.
\end{remark}

\begin{remark} 
From a view of our analysis, one can also extract a nonlinear instability result for the limiting Vlasov-Poisson equation. This result yields a stronger instability in much weaker Sobolev norms than the one obtained by Guo and Strauss in \cite{GS}, however with the additional assumption that the equilibrium $\mu$ is sufficiently smooth. It is the exact analogue in higher dimension of the 1D result proved by the first author and Hauray \cite[Theorem 3.1]{HKH}.
\end{remark}

Recalling the hyperbolic change of variables~\eqref{e-hyp}, Theorem~\ref{t-classical} will allow us to prove short time instability of the Vlasov-Maxwell system \eqref{VM}.
In this context, we can first introduce the following \emph{sharp} Penrose instability condition, ensuring instability of equilibria in the quasineutral limit.

\begin{definition}[Sharp Penrose instability condition]
\label{d-penrose}
We say that a profile $\mu(v)$ satisfies the sharp Penrose instability 
condition if there is a vector $e \in \mathbb{S}^2$ such that the function
$
\mu_e (r) = \int_{r e + e^\perp} \mu(w) \, dw
$
admits a local minimum at the point $\bar r$ and  the following inequality holds 
\begin{equation} \label{e-Penrose}
\int_{\mathbb{R}} \frac{\mu_e(r)- \mu_e(\bar r)} { |r- \bar r|^2 } \, dr >   0.
\end{equation}
If the local minimum is flat, i.e. is reached on an interval $[\bar r_1, \bar r_2]$, then~\eqref{e-Penrose} has to be satisfied for all $\bar r \in [\bar r_1, \bar r_2]$.
\end{definition}
The sharp Penrose instability condition does not directly yield a growing mode for the Vlasov-Poisson equations set on $\mathbb{T}^3 \times \mathbb{R}^3$; however the instability appears in the regime of small $\varepsilon$.

%
%We shall say here that $\mu$ is normalized if
%$$\int_{\mathbb{R}^3} \mu(v) \, dv =1, \quad \int_{\mathbb{R}^3}  \mu(v) \hat{v} \, dv =0 \text{   for all  } \varepsilon >0.$$

Our second main result reads as follows.

\begin{theorem}[Invalidity of the quasineutral limit]
\label{t-quasineutral}
Let $\mu(v)$ be a smooth, normalized equilibrium satisfying the $\delta$-condition, decaying sufficiently fast at infinity, and satisfying the sharp Penrose instability condition.
Then, for any $m,S,p>0$, there exist a 
family of smooth solutions $(f_{\varepsilon}, E_{\varepsilon}, B_{\varepsilon})_{\varepsilon>0}$ of \eqref{VM}, with $f_\varepsilon\ge 0$ and a sequence of times $t_\varepsilon = \mathcal{O}(\varepsilon |\log \varepsilon|)\to 0$,  such that 
\begin{equation}
\| (1+| v|^2)^{\frac m2}({f_{\varepsilon}}_{\vert_{t=0}}- \mu)\|_{H^S(\mathbb{T}^3\times \mathbb{R}^3)} \leq \varepsilon^p,
\end{equation}
but 
\begin{equation}
\liminf_{\varepsilon \rightarrow 0}  \left\| f_\varepsilon(t_\varepsilon) - \mu\right\|_{L^2(\mathbb{T}^3\times \mathbb{R}^3)} >0,
\end{equation}
\begin{equation}
\liminf_{\varepsilon \rightarrow 0}  \left\| \rho_\varepsilon(t_\varepsilon) - 1 \right\|_{L^2(\mathbb{T}^3)} >0,
\quad
\liminf_{\varepsilon \rightarrow 0}  \left\| j_\varepsilon(t_\varepsilon) \right\|_{L^2(\mathbb{T}^3)} >0,
\end{equation}
\begin{equation}
\liminf_{\varepsilon \rightarrow 0}   \varepsilon \left\| E_\varepsilon(t_\varepsilon) \right\|_{L^2(\mathbb{T}^3)} >0,
\end{equation}
in which $\rho_\varepsilon, j_\varepsilon$ are defined as in \eqref{def-rhoje}.
\end{theorem}

The following of this paper is dedicated to the proofs of Theorem~\ref{t-classical} and~\ref{t-quasineutral}.

\section{Instability in the classical limit}
\label{s-classical}

We study the instability of Vlasov-Maxwell systems in the classical limit $\varepsilon \to 0$: 
\begin{equation}\label{clVM1} 
\left \{ \begin{aligned}
\partial_s f^\varepsilon + \hat{v} \cdot \nabla_y f^\varepsilon + (E^\varepsilon + \varepsilon \hat{v} \times B^\varepsilon)\cdot \nabla_v f^\varepsilon & =0,
\\
\varepsilon \partial_s B^\varepsilon + \nabla_y \times E^\varepsilon = 0, \qquad \nabla_y \cdot E^\varepsilon  &= \rho^\varepsilon - 1,
\\
- \varepsilon \partial_s E^\varepsilon + \nabla_y \times B^\varepsilon = \varepsilon j, \qquad \nabla_y \cdot B^\varepsilon & =0, \end{aligned}
\right.
\end{equation}
where $$
\begin{aligned}
\hat{v} &= \frac{v}{\sqrt{1+  {\varepsilon^2 |v|^2}}}, \qquad \rho^\varepsilon (t,x)= \int_{\mathbb{R}^3} f^\varepsilon \; dv, 
\\ j^\varepsilon (t,x) &= \int_{\mathbb{R}^3}\frac{v}{\sqrt{1+  {\varepsilon^2 |v|^2}}} f^\varepsilon\; dv. 
\end{aligned}
$$
Let $\mu(v)$ be a smooth, normalized equilibrium of \eqref{clVM1}. %, therefore with zero electric and magnetic fields. %In particular, we have the charge density and moment $\rho(\mu) = \int \mu \; dv = 1$ and $j(\mu) = \int v \mu \; dv = 0$ at the equilibrium. 
We set $f^\varepsilon = \mu + f$, $E^\varepsilon = E$ and $B^\varepsilon = B$. The perturbation $(f,E,B)$ solves 
\begin{equation}\label{clVM-pert}
\partial_s f + \hat{v} \cdot \nabla_y f + (E + \varepsilon \hat{v} \times B)\cdot \nabla_v (\mu + f) = 0.
\end{equation}
The fields $E$ and $B$ are constructed through the electromagnetic potentials: 
\begin{equation}\label{potentials-cl}
 E = - \nabla \phi - \varepsilon \partial_s A, \qquad B = \nabla \times A,\end{equation}
 with $A$ satisfying the \emph{Coulomb gauge}
 $$
 \nabla \cdot A = 0.
 $$
The scalar and vector potentials $\phi, A$ solve
\begin{equation}\label{ellip-pert} 
-\Delta \phi= \rho(f), \qquad \varepsilon^2 \partial_s^2 A - \Delta A= \varepsilon j(f)- \varepsilon \partial_s \nabla \phi,\end{equation}
where we set 
$$\rho (f) : = \int f(s,y,v)\; dv, \quad j (f) : = \int \hat{v} f(s,y,v)\; dv.$$

As discussed in the introduction, the instability comes from that of the underlying Vlasov-Poisson system. In other words, we rely on a growing mode of the linearization of Vlasov-Poisson around $\mu$ in order to build an approximate growing solution to the nonlinear perturbation systems \eqref{clVM-pert}-\eqref{ellip-pert}. 

Let us start by introducing
the linearized Vlasov-Poisson operator, acting on functions $f$ with zero mean:
\begin{equation}\label{def-VP} P f: = \partial_s f - L_0 f, \quad L_0 f := - v \cdot \nabla_y f  - \nabla_v \mu \cdot \nabla_y \Delta_y^{-1} \rho (f). \end{equation}
%in which the operator $\nabla \Delta^{-1}$ is uniquely defined, since the inverse Laplacian $\Delta^{-1}$ is well-defined up to a constant on the torus. 
Note that $\rho (f)  = \int f(s,y,v)\; dv$ is a function of $(s,y)$ with zero mean in $y$.

We denote by $H ^k$ the usual Sobolev space of functions in $y$ over $\mathbb{T}^3$ (or in $v$ over $\mathbb{R}^3$) with all partial derivatives up to order $k$ having finite $L^2$ norms, and denote by $H_m^{n}$ the function space consisting of functions in $y$ and $v$ so that the norm 
$$ \| f\|_{H^{n}_m}: = \sum_{|\alpha| +|\beta| \leq n}\|  \langle v \rangle^{m} \partial_y^\alpha  \partial_v^\beta f\|_{L^2}  $$
is finite, with $\langle v \rangle: = \sqrt {1+|v|^2}$ and $m,n\ge 0$. 

Defining the domain of $L_0$ as
$$
D(L_0) = \left\{ f \in H^{n}_m, \, L_0 f \in H^{n}_m, \, \iint f \, dv dy =0 \right\},
$$
where $n,m$ are large enough, we recall (see for instance \cite{Degond}) that the unstable spectrum is made only of point spectrum. Moreover,
 there is at least one unstable eigenvalue associated to the largest positive real part $\Re \lambda$ among all elements of the spectrum; we pick one, denote it by $\lambda_0$, and refer to it as the maximal unstable eigenvalue.

\subsection{Grenier's iterative scheme}

%TODO: \cite{GS} ??

We shall construct an approximate solution $f_\mathrm{app}$ to the nonlinear problem \eqref{clVM-pert}, following the methodology introduced by Grenier \cite{G} for the study of instability of boundary layers in the inviscid limit of the Navier-Stokes equations.
%Since the approximation is constructed from the growing mode of the linearized Vlasov-Poisson, we start by introducing
%the linearized Vlasov-Poisson operator: 
%\begin{equation}\label{def-VP} P f: = \partial_s f + v \cdot \nabla f  + \nabla_v \mu \cdot \nabla \Delta^{-1} \rho (f), \end{equation}
%in which the operator $\nabla \Delta^{-1}$ is uniquely defined, since the inverse Laplacian $\Delta^{-1}$ is well-defined up to a constant on the torus. Here, $\rho (f) : = \int f(s,y,v)\; dv$ is a function of $(s,y)$. 
In view of the nonlinear equation \eqref{clVM-pert} and of the linearized Vlasov-Poisson operator~\eqref{def-VP}, we introduce 
$$
\begin{aligned}
S(f)&:= -  [\partial_s A(f) - \hat v \times (\nabla \times A(f))]\cdot \nabla_v  \mu, \\
T(f) &:= - \frac{ |v|^2}{\sqrt{1+ \varepsilon^2 |v|^2}(1+ \sqrt{1+ \varepsilon^2 |v|^2})} v\cdot \nabla_y f,
\\
 Q(f, g)&: =  - \nabla \phi(f) \cdot \nabla_v g  - \varepsilon [\partial_s A(f) - \hat v \times (\nabla \times A(f))]\cdot \nabla_v g,
 \end{aligned}$$
in which the potentials $(\phi(f), A(f))$ solve the elliptic and wave problem \eqref{ellip-pert} with the source associated to $f$. In what follows, the wave equation is solved with zero initial data: 
$$
A_{\vert_{s=0}} = \partial_s A_{\vert_{s=0}} = 0.$$ 
Finding a solution to the nonlinear problem \eqref{clVM-pert} is equivalent to solving the following symbolic equation: 
 \begin{equation}\label{def-Rf} R(f) : = Pf  + \varepsilon S(f) + \varepsilon^2 T(f) + Q(f,f) = 0.
\end{equation}
We first point out that, as it will become clear from estimate~\eqref{A-est},  the term $\varepsilon \partial_s A$ appearing in $\varepsilon S(f)$ is not small as compared to $f$, and will not be treated as a perturbation. It turns out that we can extract the leading part of $\varepsilon \partial_s A$ into the Vlasov-Poisson operator (see \cite[Section 2]{NStr}, for a similar use of this idea). 
Indeed, using the definition of $P(f)$, we write 
$$\begin{aligned}
S(f) &= -  [\partial_s A - \hat v \times (\nabla \times A)]\cdot \nabla_v  \mu 
\\
&= - P(A\cdot \nabla_v \mu) + \hat v\cdot \nabla (A\cdot \nabla_v \mu) + \nabla_v \mu \cdot \nabla \Delta^{-1} (\rho (A\cdot \nabla_v \mu)) + [\hat v \times (\nabla \times A)]\cdot \nabla_v  \mu 
\\
&= - P(A\cdot \nabla_v \mu) +\tilde S(f),
\end{aligned}$$
in which $ \rho (A\cdot \nabla_v \mu) = 0$ by definition, and we have set 
\begin{equation}\label{def-tS}
\tilde S(f) :=  (\hat v\cdot \nabla) A\cdot \nabla_v \mu + [\hat v \times (\nabla \times A)]\cdot \nabla_v  \mu ,
\end{equation}
Thus, the problem \eqref{def-Rf} is equivalent to 
\begin{equation}\label{Rf-new}  R(f) = P[f -  \varepsilon A \cdot \nabla_v \mu] + \varepsilon  \tilde S(f) + \varepsilon^2 T(f) +  Q(f, f) = 0.
\end{equation}

%TODO: say that we explain later real parts

It is now straightforward to (formally) construct an approximate solution $f_\mathrm{app}$ so that the error $R(f_\mathrm{app})$ is arbitrarily small. We start the construction with 
\begin{equation}\label{gmode}(g_1, \phi_1) =  e^{\lambda_0 s} (\hat g_1, \hat \phi_1)(y,v)
\end{equation}
 %to be the real part of 
 to be a growing mode associated to the maximal unstable eigenvalue $\lambda_0$, constructed as in Lemma \ref{lem-unmode}; we choose $(\hat g_1, \hat \phi_1)$  so that 
\begin{equation}
\label{e-g1}
 \hat{g}_1(y,v) = r_1 e^{i k_0 \cdot y} \,   \frac{1}{|k_0|^2} \frac{ \nabla_v \mu \cdot k_0}{k_0\cdot (v-\omega_0)} 
\end{equation} 
 with $r_1>0$ large enough in order to ensure
\begin{equation}
\label{e-theta}
\| \hat g_1 \|_{H^{-\SS}} \geq 2 \theta_0,
\end{equation} 
for $S'>0$ as in Theorem~\ref{t-classical}, and some $\theta_0>2$. 
By construction, $(g_1, \phi_1)$ solves the linearized Vlasov-Poisson system: 
$$P(g_1)=0, \qquad  \phi_1 := -\Delta^{-1} \rho(g_1).
$$
We assume for the moment that $(g_1, \phi_1)$ is real and shall explain later how to deal with the general case.

We then solve for $(f_1, A_1)$ satisfying 
\begin{equation}\label{g-to-f} f_1 - \varepsilon A_1 \cdot \nabla_v \mu = g_1, \qquad \varepsilon^2 \partial_s^2 A_1 - \Delta A_1 = \varepsilon(j(f_1) -\partial_s \nabla \phi_1), \qquad {A_1}_{\vert_{s=0}} = \partial_s {A_1}_{\vert_{s=0}} = 0.\end{equation}
We shall justify later (see Lemma~\ref{lem-gf}) why this system indeed has a solution.  

\bigskip

Let $p \in \mathbb{N} \setminus \{0,1\}$.
Observe that $ \varepsilon^p  f_1$ approximately solves the nonlinear equation \eqref{Rf-new}, leaving an error 
$$ R(\varepsilon^p  f_1): =   \varepsilon^{p+1} \tilde S(f_1)  + \varepsilon^{2p} Q(f_1, f_1) + \varepsilon^{p+2} T(f_1),$$
which formally is of order $\mathcal{O}(\varepsilon^{p+1})$. To obtain an error with higher order,  we introduce $(f_2, \phi_2, A_2)$ so that $(g_2, \phi_2)$, with $g_2 = f_2 - \varepsilon A_2 \cdot \nabla_v \mu$, solve%s the linearized non-homogeneous Vlasov-Poisson system: 
\begin{equation}\label{lin-f2} 
\begin{aligned}
P(g_2) =  - \tilde S(f_1), \qquad \phi_2 = -\Delta^{-1} \rho(g_2),\qquad {g_2}_{\vert_{s=0}} = 0, \\
\varepsilon^2 \partial_s^2 A_2 - \Delta A_2 = \varepsilon(j(f_2) -\partial_s \nabla \phi_2), \qquad {A_2}_{\vert_{s=0}} = \partial_s {A_2}_{\vert_{s=0}} = 0.
\end{aligned}
 \end{equation}
%in which again $A_2$ solves the wave equation \eqref{ellip-pert} with a source given by $j(g_2)$ and $\phi_2$, and  zero initial data. 
It follows directly that $\varepsilon^p f_1 + \varepsilon^{p+1} f_2$ approximately solves the nonlinear equation \eqref{Rf-new}, with a better error: 
$$ 
\begin{aligned}
R(\varepsilon^p f_1+ \varepsilon^{p+1} f_2)&:=  \varepsilon^{p+2}\Big[ [S(\tilde f_2) + T(f_1)\Big] + \varepsilon^{p+3}  T(f_2)   \\
&\qquad + \varepsilon^{2p} Q(f_1, f_1) +  \varepsilon^{2p+1} \Big[ Q(f_1, f_2) + Q(f_2, f_1)]\Big]  +  \varepsilon^{2p+2}  Q(f_2,f_2).
\end{aligned}
$$
Inductively, we construct 
\begin{equation}\label{f-app} f_\mathrm{app} =  \sum_{k=1}^N \varepsilon^{p+k-1} f_k,\end{equation}
in which $(f_k, \phi_k, A_k)$, $k\ge 3$, are defined as the unique solution to the linear problems: 
\begin{equation}\label{lin-fk} 
\left\{ \begin{aligned}
P(g_k) &=-\tilde S(f_{k-1})  - T(f_{k-2})- \sum_{\ell=1}^{k-1} Q(f_\ell, f_{k+1-p-\ell}), \qquad \phi_k = -\Delta^{-1} \rho(g_k),\qquad {g_k}_{\vert_{s=0}} = 0
\\
f_k &- \varepsilon A_k \cdot \nabla_v \mu = g_k , \qquad \varepsilon^2 \partial_s^2 A_k - \Delta A_k = \varepsilon(j(f_k) -\partial_s \nabla \phi_k), \qquad {A_k}_{\vert_{s=0}} = \partial_s {A_k}_{\vert_{s=0}} = 0,
\end{aligned}\right.
\end{equation}
with the convention that $f_{-j}= 0$ for $j \in \mathbb{N}$.
%Similarly, $A_k$ is constructed as the unique solution to the wave problem as in \eqref{ellip-pert}  with zero initial data. 
The error of this approximation can be computed as 
$$ R(f_\mathrm{app}) = - \varepsilon^{N+p} (\tilde S(f_N) + T(f_{N-1}) )- \varepsilon^{N+p+1} T(f_{N})  -\sum_{k + \ell > N+1-p ; \, 1\le k,\ell\le N-1} \varepsilon^{2p + k + \ell -2} Q(f_k,f_\ell),$$
which is of order $\mathcal{O}(\varepsilon^{N+p})$ or higher. 

%We shall take the integer $N$ so that  
%\begin{equation}\label{def-N}  \lambda_1 \le (N+p-1)\Re \lambda_0.\end{equation}

In the following subsections, we shall derive relevant estimates on each $f_k$ in the approximate solution $f_\mathrm{app}$ and deduce appropriate bounds on the approximation. In the proof,  $C$ (with various subscripts) 
shall always refer to a positive constant which can change from line to line but does not depend on $\varepsilon$.

\subsection{Linear estimates}
In this section, we obtain bounds on the profile solutions $f_k$, solving \eqref{g-to-f} and \eqref{lin-fk}. We start by studying the linear semigroup $e^{L_0s} $. %, defined as the solution operator to the linearized Vlasov-Poisson problem $P(f) = 0$ with initial data $f(0) = f_0$. 
%Here, $L_0$ denotes the linearized operator:
%$$L_0 f: = - v \cdot \nabla f  - \nabla_v \mu \cdot \nabla \Delta^{-1} \rho (f).$$
%We denote by $H ^n$ the usual Sobolev space of functions in $y$ over $\mathbb{T}^3$ with all partial derivatives up to $n$ having finite $L^2$ norms, and denote by $H^{n}_m$ the function space consisting of functions in $y$ and $v$ so that the norm 
%$$ \| f\|_{H^{n,r}_m}: = \sum_{j=0}^r \sum_{|\beta| = j} \Big( \int_{\mathbb{R}^3} \langle v \rangle^m \| \partial_v^\beta f\|_{H^n }^2 \; dv \Big)^{1/2} $$
%is finite, with $\langle v \rangle: = \sqrt {1+|v|^2}$ and $m,n,r\ge 0$. 
We have the following sharp semigroup bounds with losses of derivatives and integrability in $v$. Here,  sharp 
refers to the growth in time, in the sense that the optimal growth one could hope for would be in $e^{\Re \lambda_0 t}$ and we reach $e^{(\Re \lambda_0+\beta) t}$ for all $\beta>0$.

\begin{proposition}[Sharp bounds on the solution operator]\label{prop-exp} Let $\mu(v)$ be a smooth unstable equilibrium of Vlasov-Poisson system which decays sufficiently fast as $v \to \infty$, and let $\lambda_0$ be the maximal unstable eigenvalue. Let $n, r\ge 0, m\ge 2,$ and $h$ in $H_{m+2}^{n+2}
$. Then, $f = e^{L_0s} h$ is well-defined as the solution of the linearized Vlasov-Poisson problem $P(f) = (\partial_s - L_0) f =0$ with the initial data $h$. Furthermore, there holds 
 \begin{equation}
\label{bound-exp}  
\|  e^{L_0s} h \|_{H_m^{n} } \le C_\beta e^{(\Re \lambda_0+\beta) s} \|  h\|_{H_{m+2}^{n+2} }, \qquad \forall s\ge 0, \quad \forall \beta>0,
\end{equation}
 for some constant $C_\beta$ depending on $\mu$ and $\beta$. 
 \end{proposition}
 \begin{proof} Let $n \ge 0, m\ge 2$. For each $f$ in ${H^{n}_2}$ with $\iint f \, dy dv =0$, we denote $\phi = - \Delta^{-1} \rho (f)$. The standard elliptic theory yields 
\begin{equation}\label{Pellip-est} \| \nabla \phi\|_{H^{n+1} } \le C_0 \|\rho(f)\|_{H^n } \le C_0 \|\langle v\rangle^2  f\|_{H^{n}_y L^2_v},\end{equation}
in which $C_0$ is some universal constant. We consider the resolvent equation:
$$ (\lambda - L_0)f = h, \qquad   \, \Re \lambda \geq 0.$$
Standard $L^2$ energy estimates yield at once
$$ \Re \lambda \| \langle v\rangle^m f\|_{L^2} \le \| \langle v\rangle^m \nabla_v \mu \|_{L^2 (\mathbb{R}^3)} \| \nabla \phi\|_{L^2} + \| \langle v\rangle^m h\|_{L^2}.$$  
Therefore we deduce the following weighted $L^2$ resolvent bound, for $\Re \lambda > \gamma_{0,m}:= C_0 \| \langle v\rangle^m \nabla_v \mu \|_{L^2 (\mathbb{R}^3)}$,
$$
\| \langle v\rangle^m (\lambda - L_0)^{-1} h\|_{L^2} \le \frac{1}{\Re \lambda - \gamma_{0,m} } \| \langle v\rangle^m h\|_{L^2}
$$
Similarly, higher derivatives estimates are obtained in the similar fashion, since we observe that we have for $\alpha, \beta \in \mathbb{N}^3$,
$$
\begin{aligned}
% \lambda \partial_y^\alpha f + v \cdot \nabla_y \partial_y^\alpha f  - \nabla_v \mu \cdot \nabla \partial_y^\alpha \phi(f) &=0
%\\
 \lambda \partial_v^\beta \partial_y^\alpha f + v \cdot \nabla_y \partial_v^\beta \partial_y^\alpha f  - \nabla_v \partial_v^\beta \mu \cdot \nabla \partial_y^\alpha \phi(f)  + [\partial_v^\beta, v\cdot \nabla_y] \partial_y^\alpha f &= \partial_v^\beta \partial_y^\alpha h
\end{aligned}
 $$
in which %$\partial_y^\alpha= \partial_{x_1}^{n_1}\partial_{x_2}^{n_2}\partial_{x_3}^{n_3}$, $\partial_v^\beta= \partial_{v_1}^{r_1}\partial_{v_2}^{r_2}\partial_{v_3}^{r_3}$ and  
$[\partial_v^\beta, v\cdot \nabla_y] = \partial_v^\beta (v\cdot \nabla_y ) - v\cdot \nabla_y \partial_v^\beta$. This identity, together with the elliptic estimate \eqref{Pellip-est}, first yields
$$
\begin{aligned}
\Re \lambda \| \langle v\rangle^m \partial_y^\alpha f \|_{L^2} 
&\le \| \langle v\rangle^m\nabla_v  \mu \cdot \nabla \partial_y^\alpha \phi(f)\|_{L^2} +  \| \langle v\rangle^m \partial_y^\alpha h \|_{L^2} 
\\& \le C_0 \| \langle v\rangle^m\nabla_v \mu\|_{L^2(\mathbb{R}^3)}  \|\langle v\rangle^2 f\|_{H^{|\alpha|-1 }_y L^2_v}  +  \| \langle v\rangle^m \partial_y^\alpha h \|_{L^2}.
\end{aligned}
 $$
By a straightforward induction, we find
\begin{equation}\label{Hn-exp}
\begin{aligned}
\Re \lambda \|\langle v\rangle^m  f \|_{H_y^{|\alpha|} L^2_v}  \le C'_0 \| \langle v \rangle^2 f\|_{L^2}  + C'_0  \| \langle v\rangle^m  h \|_{H_y^{|\alpha|} L^2_v}.%\qquad n \ge 0,\qquad m\ge 6,
\end{aligned}
 \end{equation}
%for all $\lambda$ so that $\Re \lambda \ge \theta>0$. Here, $\theta$ is arbitrary, and $C_\theta$ may depend on $\theta$, but does not depend on $\lambda$.
Similarly, we have
$$
\begin{aligned}
\Re \lambda \| \langle v\rangle^m \partial_v^\beta \partial_y^\alpha f \|_{L^2} 
&\le \| \langle v\rangle^m\nabla_v \partial_v^\beta \mu \cdot \nabla \partial_y^\alpha \phi(f)\|_{L^2}  + \| \langle v\rangle^m[\partial_v^\beta, v\cdot \nabla_y] \partial_y^\alpha f\|_{L^2} + \| \langle v\rangle^m \partial_v^\beta \partial_y^\alpha h \|_{L^2} 
\\
&\le C_0 \| \langle v\rangle^m\nabla_v \partial_v^\beta \mu\|_{L^2(\mathbb{R}^3)}  \| \langle v\rangle^2  f\|_{H^{|\alpha|-1}_y L^2_v}  \\
&\qquad \qquad+ C_0 \| \langle v\rangle^m f \|_{ H_y^{|\alpha|+1}  H_v^{|\beta|-1}} + \| \langle v\rangle^m \partial_v^\beta \partial_y^\alpha h\|_{L^2}.
\end{aligned}
 $$
Again, by induction, this proves 
$$\begin{aligned}
\Re \lambda \| \langle v\rangle^m \partial_v^\beta \partial_y^\alpha f \|_{L^2} 
&\le C'_0 \| \langle v\rangle^2  f\|_{H^{|\alpha|+|\beta|-1}_y L^2_v}  + C'_0 \| \langle v\rangle^m f\|_{H^{|\alpha|+|\beta|}_y L^2_v} + C'_0 \|  h\|_{H_m^{n}} ,
\end{aligned}
$$
which together with the above bound \eqref{Hn-exp} on $\partial_y^{|\alpha| + |\beta|} f$ gives, considering all multi-indices $\alpha, \beta$ such that $|\alpha| + |\beta| \leq n$, that there exists  $\gamma_{n,m}, C_{n,m}>0$ such that
\begin{equation}\label{high-bound}
\begin{aligned}
\Re \lambda \| f \|_{H_m^{n}} 
&\le \gamma_{n,m} \| \langle v \rangle^2 f\|_{L^2}   + C_{n,m} \|  h\|_{H_m^{n}} ,
\end{aligned}
 \end{equation}
for $m\ge 2$. In particular, this proves that
$$\|  (\lambda - L_0)^{-1} h\|_{H_m^{n}} \le \frac{C_{n,m}}{\Re \lambda - \gamma_{n,m}} \| h\|_{H_{m}^{n}},$$
for some positive constant $C_{\gamma_{n,m}}$, and for all $\lambda \in \mathbb{C}$ so that $\Re \lambda >\gamma_{n,m}$. The classical Hille-Yosida theorem then asserts that $L_0$ generates a continuous semigroup $e^{L_0s}$ on the Banach space $L^2_m$ (and hence, on $H_{m}^{n}$); see, for instance, \cite{Pazy} or \cite[Appendix A]{Zum}. In addition, there holds the following representation for the semigroup: 
\begin{equation}
\label{eLs}
e^{L_0s} h= \text{P.V. } \frac{1}{2\pi i} \int_{\gamma - i\infty}^{\gamma + i \infty} e^{\lambda s} (\lambda - L_0)^{-1} h \; d\lambda \end{equation}
for any $\gamma > \gamma_{n,m}$, where $\text{P.V. }$ denotes the Cauchy principal value.

Next, by assumption, $\lambda_0$ is an unstable eigenvalue with maximal real part, and the resolvent operator $(\lambda - L_0)^{-1}$ is in fact a well-defined and bounded operator on $H_{m}^{n}$ for all $\lambda$ so that $\Re \lambda > \Re \lambda_0$. %Hence, we can take $\gamma = \Re \lambda_0 + \beta$, for arbitrary $\beta>0$, in the above representation \eqref{eLs} for the semigroup $e^{L_0s}$. 
Let $\beta>0$ and take $\gamma = \Re \lambda_0 + \beta$.
Using the boundedness of the resolvent operator, we obtain at once
\begin{equation}\label{bounded-l} \Big\| \frac{1}{2\pi i} \int_{\gamma - iM}^{\gamma + i M} e^{\lambda s} (\lambda - L_0)^{-1} h \; d\lambda \Big\|_{H_m^{n}} \le C_{\beta, M} e^{(\Re \lambda_0 + \beta) s} \| h\|_{H_m^{n}},\end{equation}
for any large but fixed constant $M$. To treat the integral for large $\Im \lambda$, we observe that directly from the equation $\lambda f = L_0 f + h$, one has
$$ |\lambda| \|  f \|_{H_m^{n}} \le \| L_0 f + h \|_{H_m^{n}} \le C ( \| \nabla_yf\|_{H_{m+1}^{n}} + \| \nabla \phi\|_{H^n})+ \| h\|_{H_m^{n}}.$$
Using \eqref{high-bound}, there holds, for some $C'_{n+1,m+1}>0$,
$$
 |\lambda|   \| (\lambda - L_0)^{-1} h\|_{H_m^{n}} \le {C'_{n+1,m+1}} \| \langle v \rangle^2 f \|_{L^2} + C'_{n+1,m+1} \| h \|_{H^{n+1}_{m+1}}.
$$
We take $ \Re \lambda = \gamma$, and consider 
$$|\Im \lambda|>\frac{2}{3}C'_{n+1,m+1}.$$
We deduce 
\begin{equation}\label{large-l} \| (\lambda - L_0)^{-1} h\|_{H_m^{n}} \le \frac{C_\beta}{|\Im \lambda|} \| h\|_{H_{m+1}^{n+1}}.\end{equation}
Now, to estimate the integral for large $|\Im \lambda|$, we may write $$(\lambda - L_0)^{-1} h = \frac1\lambda (\lambda - L_0)^{-1} L_0h + \frac h \lambda.$$
Thus, with $\gamma = \Re \lambda_0 + \beta$, we get
$$ 
\begin{aligned}
\text{P.V. } &\frac{1}{2\pi i} \int_{\{|\Im \lambda |\ge M\}}e^{\lambda s} (\lambda - L_0)^{-1} h \; d\lambda  
\\&= \text{P.V. }\frac{1}{2\pi i} \int_{\{|\Im \lambda|\ge M\}} e^{\lambda s} (\lambda - L_0)^{-1} \frac{L_0h}{\lambda} \; d\lambda 
+ \text{P.V. } \frac{1}{2\pi i} \Big[ \int_{\gamma - i\infty}^{\gamma + i \infty}  - \int_{\{|\Im \lambda|< M\}} \Big] e^{\lambda s} \frac{h}{\lambda} \; d\lambda 
\end{aligned}$$
in which the second integral on the right-hand side is equal to $h$, whereas the last integral is bounded by $C_0e^{\gamma s}h$. We take $M\geq \frac{2}{3}C'_{n+1,m+1}$ so that the bound \eqref{large-l} holds. This yields
$$
\begin{aligned}
 \Big \| \int_{\{|\Im \lambda|\ge M\}} e^{\lambda s} (\lambda - L_0)^{-1} \frac{L_0h}{\lambda} \; d\lambda \Big\|_{H_m^{n}} 
&\le C_{\beta, M} e^{\gamma s} \| L_0 h\|_{H_{m+1}^{n+1}}  \int_{\{|\Im \lambda|\ge M\}} |\Im \lambda|^{-2} \; d\Im \lambda 
\\
&\le
C_{\beta, M} e^{\gamma s} \| h\|_{H_{m+2}^{n+2}}  .
\end{aligned}$$
Putting these together and combining with \eqref{bounded-l}, we get  
\begin{equation}\label{bound-els1}\| e^{L_0 s} h\|_{H_m^{n}} \le C_{\beta} e^{(\Re \lambda_0 + \beta) s} \| h\|_{H_{m+2}^{n+2}},\end{equation}
for any $\beta >0$. The lemma is proved. 
 \end{proof}

\bigskip

Next, we derive a few estimates on the electromagnetic field. Recall that the standard elliptic theory yields the elliptic estimate \eqref{Pellip-est} for $\phi = -\Delta^{-1} \rho (f)$. In addition, together with a use of the Vlasov equation $(\partial_s - L) f = 0$, the function $\partial_s \phi$ satisfies $ - \Delta \partial_s \phi = \partial_s \rho(f) = - \nabla \cdot j(f),$
which then yields 
$$ \| \partial_s  \nabla\phi\|_{H^{n} } \le C_0 \|j(f)\|_{H^n } \le C_0 \| f\|_{H^{n}_3}.$$
Similarly, the standard $H^n$ theory for the wave equation \eqref{ellip-pert} for $A = A(f)$ yields 
\begin{equation}\label{est-phiA}\begin{aligned}
 \frac12 \frac{d}{ds} \Big( \| \varepsilon \partial_s A\|_{H^n }^2 + \| \nabla A\|_{H^n }^2 \Big) 
 & \le C_0 \| \varepsilon \partial_sA\|_{H^n } \Big( \| j(f)\|_{H^n } + \|\partial_s \nabla \phi\|_{H^n }\Big)
 \\
 & \le C_0 \| \varepsilon \partial_sA\|_{H^n } \| f\|_{H^{n}_3}
 .\end{aligned}\end{equation}
 Applying the Gronwall inequality to the above, we obtain
\begin{equation}\label{A-est}\begin{aligned}
\| \varepsilon \partial_s A\|_{H^n } + \| \nabla A\|_{H^n }  & \le C_0  \int_0^s \| f(\tau )\|_{H^{n}_3}\; d\tau
 .\end{aligned}\end{equation}
In particular, by a view of the definition of the fields $E,B$ in term of the electromagnetic potentials, we get 
\begin{equation}\label{EB-est}\begin{aligned}
\| (E,B)\|_{H^n }  & \le C_0  \int_0^s \| f(\tau )\|_{H^{n}_3}\; d\tau
 .\end{aligned}\end{equation}

The following gives a link between the Vlasov-Poisson solution $g_k$ and the Vlasov-Maxwell solution $f_k$ as defined in~\eqref{lin-fk}.
\begin{lemma}\label{lem-gf}
Let $g$ be in $H_m^{n}$, for $n\ge 0$ and $m\ge 3$, with $\int g \, dv dx =0$ and let $T$ be a positive number so that $\varepsilon T \ll 1$. There exists a solution $f$ in $L^\infty(0,T; H_m^{n})$ solving the linear problem: 
\begin{equation}\label{gf} 
f - \varepsilon A \cdot \nabla_v \mu = g, \qquad \varepsilon^2 \partial_s^2 A - \Delta A= \varepsilon(j(f) + \partial_s \nabla \Delta^{-1} \rho(g)), \qquad {A}_{\vert_{s=0}} = \partial_s {A}_{\vert_{s=0}} = 0.
\end{equation}
In addition, there holds 
$$\sup_{\tau \in [0,s]} \| f(\tau)\|_{H_m^{n}} \le C_0 \sup_{\tau \in [0,s]} \| g(\tau)\|_{H_m^{n}}  + C_0 \int_0^s  \|  \langle v \rangle^3  g(\tau )\|_{L^2} \; d\tau, \qquad \forall ~ s\in [0,T],$$
for some constant $C_0$ that is independent of $\varepsilon$. 

Furthermore, in the case where $\int \hat vg \; dv dx =0$ , we have $\int_{\mathbb{T}^3} A \, dx =0$ and the following upper and lower bound on $f$:
\begin{equation}
\label{e-bound}
c_0 \sup_{\tau \in [0,s]} \| g(\tau)\|_{H_m^{n}} \le \sup_{\tau \in [0,s]} \| f(\tau)\|_{H_m^{n}} \le C_0\sup_{\tau \in [0,s]} \| g(\tau)\|_{H_m^{n}}, \qquad \forall ~ s\in [0,T],
\end{equation}
for some constants $c_0>0, C_0>0$ that are independent of $\varepsilon$. 
\end{lemma}
\begin{proof} We start by establishing a priori estimates. We first note that there holds the Poincar\'e inequality: 
$$ \Big\|A - \langle A\rangle \Big \|_{H^n } 
\lesssim  \| \nabla A \|_{H^n } \le C_0 \int_0^s \| f(\tau)\|_{H^{n}_3}\; d\tau,$$
in which $\langle A\rangle $ denotes the average of $A$ over $\mathbb{T}^3$. This yields at once 
\begin{equation}\label{avg-A} 
\begin{aligned}
\| f - \varepsilon \langle A \rangle \cdot \nabla_v \mu\|_{H_m^{n}}&\le \| f  - \varepsilon A \cdot \nabla_v \mu \|_{H_m^{n}} + \| \varepsilon A \cdot \nabla_v \mu- \varepsilon \langle A \rangle \cdot \nabla_v \mu\|_{H_m^{n}} \\
&\le \| g\|_{H_m^{n}} +  C_0 \varepsilon \int_0^s \| f(\tau)\|_{H^{n}_3}\; d\tau .
\end{aligned}
\end{equation}
Let us bound the average of $A$. Directly from the wave equation for $A$ and the equation for $f$ in terms of $g$, we get 
$$ \begin{aligned}
\varepsilon \frac{d^2}{ds^2} \langle A \rangle 
&= \langle j \rangle = \Big \langle \int \hat v g\; dv  \Big \rangle + \varepsilon \Big \langle \int \hat v (A \cdot \nabla_v \mu)\; dv \Big \rangle
\\
& = \Big \langle \int \hat v g\; dv  \Big \rangle  -   { \varepsilon D \langle A \rangle,}
\end{aligned}$$
{in which $D$ is the diagonal matrix
$$
D:= \text{Diag} \left(-\int \partial_{v_i} \mu \hat{v}_i \, dv\right),
$$
where we note that the diagonal coefficients are positive.
}
%we have used integration by parts in $v$, the fact that $\mu$ is radial, and set
%$$ r_0 : = \int_{\mathbb{R}^3} \frac{1+\frac{2}{3}\varepsilon^2|v|^2}{(1+\varepsilon^2 |v|^2)^{3/2}}  \mu \; dv >0.$$

First, we consider the case when the average of $\int \hat v g(s )\; dv$ is equal to zero. In this case, we clearly have$\langle A(s) \rangle =0$ and thus the bound \eqref{avg-A} reads 
$$ \| f (s)\|_{H^{n}_m} \le \| g\|_{H^{n}_m} +  C_0 \varepsilon \int_0^s \| f(\tau)\|_{H^{n}_3}\; d\tau.$$
By a standard fixed point argument, we obtain the existence of $f$ in $L^\infty (0,T; H^{n}_m)$ and satisfying \eqref{gf} , as long as $\varepsilon T\ll 1$ and $m\ge 3$. We straightaway deduce the upper bound of~\eqref{e-bound}, while the lower bound is obtained as follows:
$$
\begin{aligned}
\sup_{\tau \in [0,s]} \| g(\tau) \|_{H^{n}_m} &\le  \sup_{\tau \in [0,s]} \| f (\tau) \|_{H^{n}_m}  + C_0 \varepsilon \sup_{\tau \in [0,s]} \| A(\tau)  \|_{H^{n}}  \\
&\le  \sup_{\tau \in [0,s]} \| f (\tau) \|_{H^{n}_m}  + C_0 \varepsilon \int_0^s \| f(\tau)\|_{H^{n}_3}\; d\tau \\
&\le \frac{1}{c_0} \sup_{\tau \in [0,s]} \| f (\tau) \|_{H^{n}_m}.
\end{aligned}
$$

%The claimed bound on $f$ holds trivially in this case. 

In the general case when the average of $\int \hat v g(s )\; dv$ is not equal to zero, we define $\langle \varepsilon A(s) \rangle$  as the solution of the ordinary differential equation 
$$ \frac{d^2}{ds^2} \langle\varepsilon A(s) \rangle 
+ D \langle \varepsilon A (s) \rangle  = \Big \langle \int \hat v g (s)\; dv  \Big \rangle  .$$
Since the fundamental solutions to the homogeneous equation {$Y'' + D Y =0$} are bounded, the above yields at once 
$$ \varepsilon |\langle A(s)\rangle | \le \int_0^s \Big|  \Big \langle \int \hat v g(\tau )\; dv  \Big \rangle \Big| \; d\tau \le C_0 \int_0^s \| \langle v \rangle^3 g(\tau )\|_{L^2}\; d\tau.$$
We now establish a bound for $f$ satisfying 
$$ f - \varepsilon (A - \langle A \rangle) \cdot \nabla_v \mu = g + \varepsilon \langle A \rangle \cdot \nabla_v\mu .$$
The second term on the left-hand side is again a small perturbation in terms of $f$, yielding 
$$ \| f (s)\|_{H^{n}_m} \le \| g\|_{H^{n}_m} + C_0 \int_0^s \|  \langle v \rangle^3  g(\tau )\|_{L^2}\; d\tau +  C_0 \varepsilon \int_0^s \| f(\tau)\|_{H^{n}_3}\; d\tau .$$
Thus, as long as $\varepsilon T \ll 1$, %we can absorb the last term on the right-hand side  into the left-hand side, taking the sup norm for all $s\in [0,T]$, 
we can use as well a fixed point argument, yielding the existence of $f$ as well as the claimed bound.\end{proof}

\subsection{Error estimates}
Let us now give estimates on the approximate solution $f_\mathrm{app}$ and the error of the approximation $R(f_\mathrm{app})$. Let $N$ be a fixed number of the iteration in $f_\mathrm{app}$ and let 
$$m\ge 5N+3 \text{  and  } n> 3N+5/2.$$ 
%Note that there holds the product estimate: $\| uv\|_{H^n } \le \|u\|_{H^n }\|v\|_{H^n }$. 
By a view of the wave estimate \eqref{A-est} and the fact that $\mu$ decays rapidly at infinity, we have for all $n, m \geq 0$,
\begin{equation}\label{bound-SQ}
\begin{aligned}
\| \tilde S(f)(s)\|_{H^{n}_m}&\le C_0 \| \nabla A\|_{H^n } \le C_0 \int_0^s \| f(\tau)\|_{H^{n}_3}\; d\tau  
\\
\|  T(f)(s)\|_{H_{m-3}^{n-1}}&\le C_0 \| f\|_{H_{m}^{n}},
\\
 \| Q(f, g)(s)\|_{H_{m-2}^{n-1}} & \le  C_0  \|  (E, B)(s)\|_{H ^n} \|  g (s) \|_{H^{n}_m}  \le C_0 \|  g (s) \|_{H^{n}_m} \int_0^s \| f(\tau)\|_{H^{n}_3}\; d\tau.
\end{aligned}\end{equation}
%Recall that the construction starts with $g_1 = e^{\lambda_0 s} \hat g_1$, with $\| \hat g_1\|_{H^{n}_m} =1$. 
%Let $\theta_0 >0$ such that
%\begin{equation}\label{lower-bd} 
%{2}\theta_0 e^{\Re \lambda_0 s} \le \| g_1 \|_{H^{-\SS}} ,
%\end{equation}
We apply Lemma \ref{lem-gf} to the linear problem \eqref{g-to-f} for $f_1$, which yields 
\begin{equation}\label{upper-bd}
\| f_1 \|_{H^{n}_m} \le C_0 e^{\Re \lambda_0 s} .
\end{equation}
By induction, we shall prove 
\begin{align}\label{bound-fk}
\| f_k(s) \|_{H_{m-5k+5}^{n-3k+3}} + \| g_k(s) \|_{H_{m-5k+5}^{n-3k+3}} 
&\le  C_k e^{\left(1+ \frac{k-1}{p}\right)\Re \lambda_0 s}, \qquad 1 \leq k \leq N,
\end{align}
for all $s\in [0,T]$, with $\varepsilon T \ll 1$. 
The case $k=1$ is clear. Assuming the bound holds for all  $j \in \{1,\cdots,k\}$, with $k \ge 1$, we now prove the bound for $j=k+1\ge 2$. Writing a Duhamel  formula for the nonhomogeneous equation \eqref{lin-fk} on $g_{k+1}$, we find 
$$ g_{k+1} =  - \int_0^s e^{L_0(s - \tau)} \Big[\tilde S(f_{k}) + T(f_{k-1})  +\sum_{\ell=1}^{k} Q(f_\ell, f_{k+2-p-\ell})\Big](\tau)\; d\tau.$$
Using Proposition~\ref{prop-exp}, the bounds in \eqref{bound-SQ} and the induction assumption, we can estimate 
$$\begin{aligned}
 &\| g_{k+1}\|_{H_{m-5k}^{n-3k}} \\
  &\le C_\beta \int_0^s e^{(\Re \lambda_0 + \beta)(s-\tau)}\Big[\|\tilde S(f_{k}) \|_{H_{m-5k+2}^{n-3k+2}} + \|T(f_{k-1}) \|_{H_{m-5k+2}^{n-3k+2}} + \sum_{\ell=1}^{k} \|Q(f_\ell,  f_{k+2-p-\ell})\|_{H_{m-5k+2}^{n-3k+2}} \Big](\tau)\; d\tau
 \\
  &\le C_\beta \int_0^s e^{(\Re \lambda_0 + \beta)(s-\tau)}\Big[ C_k e^{(1+\frac{k-1}{p}) \Re \lambda_0 \tau}+ \sum_{\ell=1}^{k} C_{k,\ell} e^{\left(1+ \frac{\ell-1}{p}\right) \Re \lambda_0 \tau} e^{\left(1+ \frac{k+1-p-\ell}{p}\right)\Re \lambda_0 \tau}  \Big]\; d\tau
 \\
  &\le C_\beta \int_0^s e^{(\Re \lambda_0 + \beta)(s-\tau)} e^{ \left(1+ \frac{k}{p}\right) \Re \lambda_0 \tau} \; d\tau
 \\
  &\le 
  C_k e^{ \left(1+ \frac{k}{p}\right) \Re \lambda_0 s} ,
 \end{aligned}$$
 in which we have chosen $\beta = \Re \lambda_0 /p$. Next, we apply Lemma \ref{lem-gf}, yielding 
 $$\sup_{\tau \in [0,s]} \| f_{k+1}(\tau)\|_{H_{m-5k}^{n-3k}} \le C_0 \sup_{\tau \in [0,s]} \| g_{k+1}(\tau)\|_{H_{m-5k}^{n-3k}}  + C_0 \int_0^s  \|  \langle v \rangle^3  g_{k+1}(\tau )\|_{L^2}\; d\tau \le C_k e^{ \left(1+ \frac{k}{p}\right) \Re \lambda_0 s} ,$$ 
which finishes the proof of the inductive bound \eqref{bound-fk}, for all $k\ge 1$.

Using these bounds on $f_k$ as well as \eqref{bound-SQ}, we can estimate the error of the approximation:
\begin{equation}
\label{remain} \begin{aligned}
\| R(f_\mathrm{app})(s)\|_{H^{n-3N}_{m-5N}} 
&\le \varepsilon^{N+p} (\| \tilde S(f_N)\|_{H_{m-5N}^{n-3N}}+  \|  T(f_{N-1})\|_{H_{m-5N}^{n-3N}}) +
\varepsilon^{N+p+1}  \|  T(f_{N})\|_{H_{m-5N}^{n-3N}}\\
&\qquad + \sum_{k+\ell > N+1-p;\; 1\le k,\ell\le N-1} \varepsilon^{2p+k+\ell-2} \| Q(f_k,f_\ell) \|_{H^{n-3N}_{m-5N}}
\\
&\le C_0 \varepsilon^{N+p} e^{\left(1+ \frac{N-1}{p}\right) \Re \lambda_0 s} + \sum_{k+\ell > N+1-p;\; 1\le k,\ell\le N-1} \varepsilon^{2p+k+\ell-2} e^{\left(2+ \frac{k+\ell-2}{p}\right) \Re \lambda_0 s}
\\&\le C_0 \varepsilon^{N+p} e^{ \left(1+ \frac{N-1}{p}\right) \Re \lambda_0 s} + C_0 \varepsilon^{N+p} e^{ \left(1+ \frac{N}{p}\right) \Re \lambda_0 s}
\\&\le C_0 \Big( \varepsilon^{p} e^{ \Re \lambda_0 s} \Big)^{\left(1+ \frac{N}{p}\right) },
\end{aligned}
\end{equation}
for all $s\ge 0$, as long as $\varepsilon^pe^{\Re \lambda_0 s}$ remains bounded.

\subsection{Nonlinear instability}
We are ready to conclude the proof of Theorem \ref{t-classical}. The instability result now follows from a standard energy estimate. Indeed, let $(f,\phi, A)$ be the exact perturbative solution to the Vlasov-Maxwell system:  
$$
\partial_s f + \hat{v} \cdot \nabla_y f + (E + \varepsilon \hat{v} \times B)\cdot \nabla_v (\mu + f) = 0.
%\partial_s f -Lf  = -  (E + \varepsilon v \times B)\cdot \nabla_v f
$$
with the electromagnetic field solving the Maxwell equations. Consider $(f_\mathrm{app}, \phi_\mathrm{app}, A_\mathrm{app})$ the approximate solution constructed and studied in the previous sections. Let the difference be
$$(h, \phi_h, A_h) := (f - f_\mathrm{app}, \phi - \phi_\mathrm{app}, A - A_\mathrm{app}),$$
which solves 
$$\partial_s h + \hat{v} \cdot \nabla_y h  + (E_h + \varepsilon \hat{v}  \times B_h)\cdot \nabla_v (\mu + f_\mathrm{app}) +  (E_\mathrm{app} + E_h + \varepsilon \hat{v}  \times (B_\mathrm{app} + B_h))\cdot \nabla_v h = R(f_\mathrm{app}),
$$
with $h|_{t=0}=0$.
Standard weighted energy estimates yield, for $k > 7/2$,
$$ \frac12\frac{d}{ds} \| h\|_{H^{k}_3}^2 \le C_0  \| h \|_{H^{k}_3} \Big[ \|(E_h,B_h)\|_{H^k}  + \| R(f_\mathrm{app})\|_{H_3^{k}}\Big] +  \| h \|_{H^{k}_3}^2 \Big[ 1+ \|(E_h,B_h)\|_{H^k} \Big].
$$
Combining with the estimates \eqref{EB-est} on $E_h,B_h$  and with~\eqref{remain} yields at once
$$\begin{aligned}
 \frac 12 \frac{d}{ds} \Big( \| h(s)\|_{H^{k}_3}^2 +  \| (E_h,B_h)(s)\|_{H^k}^2 \Big) 
 & \le \Big[ C +  \| h (s)\|_{H^{k}_3} \Big]  \Big( \| h (s)\|_{H^{k}_3}^2 +  \| (E_h,B_h)(s)\|_{H^k}^2 \Big) 
 \\&\qquad  + C_0 \Big( \varepsilon^{p} e^{ \Re \lambda_0 s} \Big)^{2\left(1+ \frac{N}{p}\right) }.
 \end{aligned}$$ 

We now introduce 
$$ T^\varepsilon: = \sup \Big \{ s\ge 0~:~ \sup_{\tau \in [0,s]} \| h(\tau)\|_{H^{k}_3} \le \frac {\theta_0}2\varepsilon^p e^{\Re \lambda_0 \tau}\Big\},$$
where $\theta_0>2$ was fixed in~\eqref{e-theta}. 
By the standard local existence theory, we know that $T^\varepsilon >0$.

Now define 
$$\tilde T^\varepsilon:= \frac{1}{\Re \lambda_0}  |\log (\varepsilon^p \frac{\theta_0}{2}) |.$$ 
It follows that for all $s \in [0, \tilde T_\varepsilon]$, %\eqref{e-Teps} is satisfied.
  \begin{equation}
\label{e-Teps}\frac {\theta_0}2\varepsilon^p e^{\Re \lambda_0 s} \leq 1,
\end{equation}
  If $\tilde{T}^\varepsilon > T^\varepsilon$, then for any $s\in [0,T^\varepsilon]$,
the above differential inequality yields 
$$\begin{aligned}
 \frac 12 \frac{d}{ds} \Big( \| h(s)\|_{H^{k}_3}^2 +  \| (E_h,B_h)(s)\|_{H^k}^2 \Big) 
 & \le (1+C)\Big( \| h(s)\|_{H^{k}_3}^2 +  \| (E_h,B_h)(s)\|_{H^k}^2 \Big) 
 \\&\qquad  + C_0 \Big( \varepsilon^{p} e^{ \Re \lambda_0 s} \Big)^{2\left(1+ \frac{N}{p}\right) }.
 \end{aligned}$$ 
Using the Gronwall inequality and imposing $N$ large enough so that 
$$N \ge C_0, \qquad 1+C \le \Big( 1+ \frac{N}{p}\Big) \Re \lambda_0, \qquad \frac{32  C_0}{\theta_0^2} \leq {(\theta_0/2)}^{\frac{2N}{p}}, $$ 
there holds 
$$ 
\begin{aligned}
 \| h(s)\|_{H^{k}_3}^2 +  \| (E_h,B_h)(s)\|_{H^k}^2 
 &\le 2C_0 \int_0^s e^{2(1+C) (s - \tau)} \Big( \varepsilon^{p} e^{ \Re \lambda_0 \tau } \Big)^{2\left(1+ \frac{N}{p}\right) }\; d\tau
 \\
 &\le 2C_0
 \Big( \varepsilon^{p} e^{ \Re \lambda_0 s} \Big)^{2\left(1+ \frac{N}{p}\right) }
  \\&\le  \frac{ 2C_0}{(\theta_0/2)^{\frac{2N}{p}}}   \Big(  \varepsilon^p e^{\Re \lambda_0 s} \Big)^2
 \\&\leq \frac{1}{16}\theta_0^2 \Big(  \varepsilon^p e^{\Re \lambda_0 s} \Big)^2=  \left(\frac {\theta_0}4\varepsilon^p e^{\Re \lambda_0 s}\right)^2 \\
 &< \frac {\theta_0}4\varepsilon^p e^{\Re \lambda_0 s}.
 \end{aligned}$$
%Note that since $\theta_0 <1$, this is always less than $\frac {\theta_0}2\varepsilon^p e^{\Re \lambda_0 s}$.
This contradicts the definition of $T^\varepsilon$ and proves that we necessarily have $ T^\varepsilon  \ge  \tilde T^\varepsilon$.

\bigskip

Finally, recall that $f=  f_\mathrm{app} + h$. Thus, by the triangle inequality, as long as $s\in [0,\tilde T^\varepsilon]$, we get
\begin{equation}
\label{e-lower0} 
\begin{aligned}
\| f\|_{L^2} 
&\ge \| f_\mathrm{app}\|_{L^2} - \| h\|_{L^2} \\
&\ge  \| f_\mathrm{app}\|_{L^2}  - \frac {\theta_0}2\varepsilon^p e^{\Re \lambda_0 s}.
\end{aligned}
\end{equation}
We therefore need to get a lower bound on the $L^2$ norm $f_{\mathrm{app}}$. First, we have
$$
\| f_1 \|_{L^2} \geq \| g_1 \|_{L^2} -  \varepsilon \| A_1 \cdot \nabla_v \mu\|_{L^2}.
$$
 It follows from the construction that the average of $\int \hat v g_1\; dv$ is equal to zero. Indeed, by definition, we have
$$
\int \hat v  g_1\; dv dy = \int e^{i k_0 \cdot y} \, dy \int   \frac{1}{|k_0|^2} \frac{ \nabla_v \mu \cdot k_0}{k_0\cdot (v-\omega_0)} v \, dv =0.
$$
From Lemma~\ref{lem-gf}, we know that $\langle A_1 \rangle =0$, so that 
$$
\varepsilon \| A_1 \cdot \nabla_v \mu\|_{L^2} \leq \varepsilon C_0  \int_0^s \| f(\tau)\|_{H^{n}_3}\; d\tau.
$$
Recalling \eqref{upper-bd}, we end up with
$$
\varepsilon \| A_1 \cdot \nabla_v \mu\|_{L^2} \leq \varepsilon C_0 e^{\Re \lambda_0 s}.
$$
By \eqref{e-theta}, we deduce (at least for $\varepsilon>0$ small enough), 
$$
\|  f_1 \|_{L^2} \geq  \theta_0  e^{\Re \lambda_0 s}.
$$
Finally using~\eqref{bound-fk} to bound the contribution of the terms $f_k$, $k\geq 2$, we obtain
$$
\| f_\mathrm{app}\|_{L^2} \geq  \theta_0 \varepsilon^p e^{\Re \lambda_0 s} - C_N \Big(\varepsilon^p e^{\Re \lambda_0 s}\Big)^{1+\frac 1p} 
$$
and thus for  $s\in [0,\tilde T^\varepsilon]$,
\begin{equation}
\label{e-lower} 
\begin{aligned}
\| f(s)\|_{L^2} 
&\ge \theta_0 \varepsilon^p e^{\Re \lambda_0 s} - C_N \Big(\varepsilon^p e^{\Re \lambda_0 s}\Big)^{1+\frac 1p} - \frac {\theta_0}2\varepsilon^p e^{\Re \lambda_0 s} 
\\
& \ge \frac {\theta_0}2\varepsilon^p e^{\Re \lambda_0 s} \Big( 1 - \frac{2C_N}{\theta_0} \Big(\varepsilon^p e^{\Re \lambda_0 s}\Big)^{\frac 1p}\Big).
\end{aligned}
\end{equation}

Define finally $\overline T^\varepsilon :=  \frac{1}{\Re \lambda_0}  \left|\log (\varepsilon^p \frac{4C_N}{\theta_0}) \right|$.
For $s_\varepsilon := \min( \tilde T^\varepsilon, \overline T^\varepsilon)$ we end up with the lower bound 
\begin{equation}
\label{e-instabi}
\| f(s_\varepsilon)\|_{L^2(\mathbb{T}^3 \times \mathbb{R}^3)}  \ge \delta_0,
\end{equation}
with $\delta_0 =  \min \left( \frac{1}{2}, \frac{\theta_0^2}{16 C_N}\right)$. 

Similarly, for what concerns $H^{-\SS}$ instability, we can get as well
\begin{equation}
\label{e-lowerH-} 
\begin{aligned}
\| f(s_\varepsilon)\|_{H^{-\SS}} 
&\ge \varepsilon^p\|  f_1 \|_{H^{-\SS}}  - \| f_\mathrm{app} - \varepsilon^p f_1\|_{H^{-\SS}}- \| h\|_{H^{-\SS}}
\\
&\ge \varepsilon^p\| f_1 \|_{H^{-\SS}}  - \| f_\mathrm{app} - \varepsilon^p f_1\|_{L^2}- \| h\|_{L^2} 
\\&\ge \theta_0 \varepsilon^p e^{\Re \lambda_0 s_\varepsilon} - C_N \Big(\varepsilon^p e^{\Re \lambda_0 s_\varepsilon}\Big)^{1+\frac 1p} - \frac {\theta_0}2\varepsilon^p e^{\Re \lambda_0 s_\varepsilon} 
\\
& \ge 
\delta_0,
\end{aligned}
\end{equation}
 with $\delta_0$ as defined in \eqref{e-instabi}. Recalling that $f^\varepsilon-\mu= f$, this proves the first instability result of~\eqref{e-thm1}.

The instability on $\rho^\varepsilon$ and $j^\varepsilon$ is proved with similar estimates, using the \emph{weighted} $L^2$ error estimates.
The only thing to notice is that, recalling $g_1 = e^{\lambda_0 s} \hat g_1$, and by a view of~\eqref{e-g1} and of the Penrose condition~\eqref{Penrose}, we have $\rho(\hat g_1) \neq 0$ and $j(\hat g_1) \neq 0$. 
With the same arguments, we end up with
\begin{equation}
\label{e-rhoj}
\| \rho^\varepsilon(s)- 1 \|_{H^{-\SS}}  \ge \delta'_0, \quad \| j^\varepsilon(s) \|_{H^{-\SS}}  \ge \delta'_0,
\end{equation}
for some $\delta_0'>0$.

The instability on $E^\varepsilon$ then follows by a view of~\eqref{potentials-cl}. We write
$$
\| E^\varepsilon \|_{L^2}^2 = \| \nabla \phi^\varepsilon \|_{L^2}^2 + \| \varepsilon \partial_s A^\varepsilon \|_{L^2}^2 + \varepsilon \int \nabla \phi^\varepsilon \cdot \partial_s A^\varepsilon \, dx,
$$
and note that since $A^\varepsilon$ satisfies the Coulomb gauge $\nabla \cdot A^\varepsilon =0$,
$$
\int \nabla \phi^\varepsilon \cdot \partial_s A^\varepsilon \, dx = \int  \phi^\varepsilon  \,  \partial_s \nabla \cdot A^\varepsilon \, dx=0.
$$
We thus use~\eqref{e-rhoj} to get $\| \rho^\varepsilon- 1 \|_{H^{-1}}  \ge \delta'_0$ and obtain, using finally the Poisson equation~\eqref{ellip-pert},
$$
\| E^\varepsilon(s) \|_{L^2} \ge  \| \nabla \phi^\varepsilon(s) \|_{L^2} \ge C_0 \delta'_0.
$$

%TODO explain how to deal with real stuff.
\bigskip

Let us finally complete the proof of Theorem \ref{t-classical} by briefly explaining how to deal with complex eigenvalues and eigenfunctions, as well as getting non-negative distribution functions.

We assume here that $\Im \lambda_0 \neq 0$.
 Writing $\hat{g_1} = \Re  g_1 + i \Im g_1$, and assuming without loss of generality that $\Re g_1 \neq 0$,
 we set  $g_1 = \Re (e^{\lambda_0 s} \hat{g_1})$ instead of the definition~\eqref{gmode}. Then one can perform exactly the same construction and analysis, except that the lower bound for $f_{\mathrm{app}}$ in~\eqref{e-lower} is achieved for all $s$ of the form $s= \frac{2\pi k}{\Im \lambda_0}$. This is sufficient to get the instability as in~\eqref{e-instabi}.
 
 For what concerns non-negativity, we just need to notice that the $\delta$-condition and the form of the eigenfunctions (recall~\eqref{e-g1}) ensure that $\varepsilon^p |g_1|_{\vert t=0} \leq \mu$, so that the initial condition satisfies $f^\varepsilon_{\vert t=0} \geq 0$.

\section{Invalidity of the quasineutral limit}
\label{s-quasineutral}

Let $\mu(v)$ be a smooth, normalized profile satisfying the $\delta$-condition and the sharp Penrose instability condition. %(see Definition~\eqref{d-penrose}).

For any $M>0$, we shall denote 
$$\mathbb{T}_M^3 := \mathbb{R}^3/(M \mathbb{Z}\times M \mathbb{Z} \times M \mathbb{Z}).$$
We recall that for a given length $M$, the sharp Penrose instability condition does not necessarily ensure the existence of a growing mode for the linearized equations. However, the latter is true for large enough values of $M$.
This is the content of the following Proposition, taken from \cite[Proposition 3.2]{HKH}. \begin{prop} \label{prop:Penlim}
Assume that $\mu$ is a smooth homogeneous profile satisfying the sharp Penrose instability condition. There exists $M_0>0$ such that if $ 
M \geq M_0 $, then the Penrose instability condition~\eqref{Penrose} is satisfied for the equations posed on $\mathbb{T}_M^3 \times \mathbb{R}^3$.
\end{prop}
As a matter of fact, the framework of \cite{HKH} is one-dimensional but this particular result straightforwardly extends to higher dimensions. Using this Proposition, we fix 
some large enough parameter $M>0$ such that 
the linearized Vlasov-Poisson operator $L_0$  on $\mathbb{T}^3_M \times \mathbb{R}^3$ has an eigenvalue 
with positive real part. From now on, we consider the sequence $\varepsilon_k = \frac{1}{kM}$, for $k \in 
\mathbb{N}^*$,  but we forget about the $k$ subscript for readability.

As already explained in the introduction, in the high spatial frequency regime, the study of the quasineutral limit
comes down to that of the classical limit.
More precisely, we shall first consider $\varepsilon M$-periodic (in all spatial directions) solutions  to~\eqref{VM}, that is we look for solutions to the system
\begin{equation}\label{VM2} 
\left \{ \begin{aligned}
\partial_t \tilde f_\varepsilon + \hat{v}  \cdot \nabla_x \tilde f_\varepsilon + (\tilde E_\varepsilon +  \hat{v}  \times \tilde B_\varepsilon)\cdot \nabla_v \tilde f_\varepsilon & =0,
\\
 \partial_t \tilde B_\varepsilon + \nabla_x \times \tilde E_\varepsilon = 0, \qquad \varepsilon^2 \nabla_x \cdot \tilde E_\varepsilon  &= \tilde \rho_\varepsilon - 1,
\\
-  \varepsilon^2 \partial_t \tilde E_\varepsilon + \nabla_x \times \tilde B_\varepsilon =  \tilde j_\varepsilon, \qquad \nabla_x \cdot \tilde B_\varepsilon & =0, \end{aligned}
\right.
\end{equation}
 for $t\geq 0, \,  x\in \mathbb{T}_{\varepsilon M}^3, \, v\in \mathbb{R}^3$ and where 
 $$
 \hat{v} = \frac{v}{\sqrt{1+  {\varepsilon^2 |v|^2}}},
 $$
 $$
\tilde \rho_\varepsilon (t,x)= \int_{\mathbb{R}^3} \tilde f_\varepsilon \; dv, \quad \tilde j_\varepsilon (t,x)= \int_{\mathbb{R}^3}\hat{v} \tilde f_\varepsilon\; dv.
 $$
 We can then obtain a solution $(f_\varepsilon,E_\varepsilon, B_\varepsilon)$ to 
\eqref{VM} by patching $(\varepsilon M)^{-3}$ copies of $(\tilde{f}_\varepsilon,\tilde E_\varepsilon, \tilde B_\varepsilon)$.
This means, identifying $\mathbb{T}^3_M$ to $[0,M)^3$, writing
$$f_\varepsilon(t,x,v) = \tilde{f}_\varepsilon(t, x_1- j_1 \varepsilon M, x_2- j_2 \varepsilon M, x_3- j_3 \varepsilon M, v), 
$$
for all $x = (x_1, x_2,x_3)$ in 
$$\prod_{i=1}^3 [j_i \varepsilon M, \, (j_i+1)\varepsilon M), \qquad j_1, j_2, j_3= 0,\cdots, k-1.$$
Similar formulas are given for $(E_\varepsilon, B_\varepsilon)$.

We can now perform the hyperbolic change of variables $(t,x,v) \to \left(\frac{t}\varepsilon, 
\frac{x}\varepsilon,v\right)$, i.e. we consider $(g_\varepsilon, \mathbb{E}_\varepsilon, \mathbb{B}_\varepsilon)$ such 
that:
\begin{equation}
\label{e-scaling}
\tilde{f}_\varepsilon(t,x,v) = g_\varepsilon\left(\frac{t}\varepsilon,\frac{x}\varepsilon,v\right), \quad \tilde{E}_\varepsilon(t,x,v) = \frac{1}{\varepsilon}\mathbb{E}_\varepsilon\left(\frac{t}\varepsilon,\frac{x}\varepsilon\right), \quad \tilde{B}_\varepsilon(t,x,v) = \mathbb{B}_\varepsilon\left(\frac{t}\varepsilon,\frac{x}\varepsilon\right).
\end{equation}
This leads to the study of the classical limit, for $s\geq 0$, $y\in \mathbb{T}_M$, $v\in \mathbb{R}$
\begin{equation}\label{clVM2} 
\left \{ \begin{aligned}
\partial_s g_\varepsilon + \hat{v}  \cdot \nabla_y g_\varepsilon + (\mathbb{E}_\varepsilon + \varepsilon \hat{v}  \times \mathbb{B}_\varepsilon)\cdot \nabla_v g_\varepsilon & =0,
\\
\varepsilon \partial_s \mathbb{B}_\varepsilon + \nabla_y \times \mathbb{E}_\varepsilon = 0, \qquad \nabla_y \cdot \mathbb{E}_\varepsilon  &= \int g_\varepsilon \, dv - 1
\\
- \varepsilon \partial_s \mathbb{E}_\varepsilon + \nabla_y \times \mathbb{B}_\varepsilon = \varepsilon \int \hat{v} g_\varepsilon \, dv  , \qquad \nabla_y \cdot \mathbb{B}_\varepsilon & =0. \end{aligned}
\right.
\end{equation}

We apply Theorem~\ref{t-classical} ($M$ being considered as a fixed transparent parameter).
Let $S,N \in \mathbb{N}^*$ and $p \in \mathbb{N}^*$, such that $p> S+ N$. We take $k=0$. By 
Theorem \ref{t-classical}, we find for all $\varepsilon \in (0,1]$ a solution $(g_\varepsilon, 
\mathbb{E}_\varepsilon, \mathbb{B}_\varepsilon)$ to \eqref{clVM2} 
with $g_\varepsilon\ge 0$, such that 
\begin{equation}
\| (1+| v|^2)^{\frac m2} ({g_{\varepsilon}}_{\vert_{s=0}}- \mu)\|_{H^S(\mathbb{T}_M^3\times \mathbb{R}^3)} \leq \varepsilon^p,
\end{equation}
but there is a sequence of times $s_\varepsilon = \mathcal{O}(|\log \varepsilon|)$ such that
\begin{equation}
\liminf_{\varepsilon \rightarrow 0}  \left\| g_\varepsilon(s_\varepsilon) - \mu\right\|_{L^2(\mathbb{T}^3_M \times \mathbb{R}^3)} >0,\quad \liminf_{\varepsilon \rightarrow 0}  \left\| \mathbb{E}_\varepsilon(s_\varepsilon) \right\|_{L^2(\mathbb{T}^3_M)} >0,
\end{equation}
\begin{equation}
\liminf_{\varepsilon \rightarrow 0}\left\| \rho_{g_\varepsilon}(s_\varepsilon) - 1 \right\|_{L^2(\mathbb{T}^3_M)} >0,
\quad
\liminf_{\varepsilon \rightarrow 0}  \left\| j_{g_\varepsilon}(s_\varepsilon) \right\|_{L^2(\mathbb{T}^3_M)} >0,
\end{equation}
with the notation $ \rho_{g_\varepsilon}(t)= \int_{\mathbb{R}^3} g_\varepsilon(t,x,v) \,dv$ and $ j_{g_\varepsilon}(t)= \int_{\mathbb{R}^3} \hat{v} g_\varepsilon(t,x,v) \,dv$.

Next, a consequence of of the change of variable~\eqref{e-scaling} and of the $\varepsilon M$-periodicity of $f_\varepsilon$ is that:
\begin{equation} \label{e-back} 
\begin{aligned}
 &\| (1+| v|^2)^{\frac m2}({f_\varepsilon}_{\vert t=0}  - \mu) \|_{H^S(\mathbb{T}^3 \times \mathbb{R}^3)} \le C \frac{\varepsilon^{-S}}{M^3} \| {g_{\varepsilon}}_{\vert_{s=0}} -\mu  \|_{H^S(\mathbb{T}_M^3 \times \mathbb{R}^3)}  %\quad \text{for } \; S  \in \mathbb{N}, 
 \\
 & \| f_\varepsilon(t) -1 \|_{L^2(\mathbb{T}^3 \times \mathbb{R}^3)} = \frac{1}{M^3} \|  {g_\varepsilon}\left({t}/{\varepsilon}\right) -1 \|_{L^2(\mathbb{T}_M^3 \times \mathbb{R}^3)}, \quad \varepsilon\| E_\varepsilon(t)  \|_{L^2(\mathbb{T}^3)} = \frac{1}{M^3} \|  \mathbb{E}_\varepsilon \left({t}/{\varepsilon}\right)  \|_{L^2(\mathbb{T}_M^3)}, \\
&\| \rho_\varepsilon(t) -1 \|_{L^2(\mathbb{T}^3)} = \frac{1}{M^3} \|  \rho_{g_\varepsilon}\left({t}/{\varepsilon}\right) -1 \|_{L^2(\mathbb{T}_M^3)},  
\quad \| j_\varepsilon(t) -1 \|_{L^2(\mathbb{T}^3 )} = \frac{1}{M^3} \|  j_{g_\varepsilon}\left({t}/{\varepsilon}\right)  -1 \|_{L^2(\mathbb{T}_M^3)}.
\end{aligned}
\end{equation}
We set $t_\varepsilon := \varepsilon s_\varepsilon  = \mathcal{O}(\varepsilon |\log \varepsilon|)$ and  deduce, at least for $\varepsilon$ small enough, 
\begin{align*}
&\| (1+| v|^2)^{\frac m2}({f_\varepsilon}_{\vert t=0}  - \mu) \|_{H^S(\mathbb{T}^3 \times \mathbb{R}^3)} \leq  \frac{1}{M^3} \,
\varepsilon^{p-S} \leq \varepsilon^N, \\
& \liminf_{\varepsilon \rightarrow 0}  \left\| f_\varepsilon(t_\varepsilon) - \mu\right\|_{L^2(\mathbb{T}^3\times \mathbb{R}^3)} >0,\quad \liminf_{\varepsilon \rightarrow 0}   \varepsilon \left\| E_\varepsilon(t_\varepsilon) \right\|_{L^2(\mathbb{T}^3)} >0, \\
&\liminf_{\varepsilon \rightarrow 0}  \left\| \rho_\varepsilon(t_\varepsilon) - 1 \right\|_{L^2(\mathbb{T}^3)} >0,
\quad
\liminf_{\varepsilon \rightarrow 0}  \left\| j_\varepsilon(t_\varepsilon) \right\|_{L^2(\mathbb{T}^3)} >0,
\end{align*}
which proves Theorem~\ref{t-quasineutral}.

\bigskip

\noindent {\bf Acknowledgements.} We are grateful to Pennsylvania State University and \'Ecole polytechnique for hospitality during the preparation of this work. TN was partly supported by the NSF under grant DMS-1405728 and by a two-month visiting position at \'Ecole polytechnique during Spring 2015.

We thank the referee for several suggestions that helped to improve the presentation of the paper and especially for correcting an earlier incorrect interpretation of Theorem~\ref{t-classical}.

{We are also indebted to Aymeric Baradat for pointing out that no radial profile in dimension $3$ satisfies the Penrose instability criterion.}

\bibliographystyle{plain}
\bibliography{eMHD}

\begin{thebibliography}{10}

\bibitem{ASA}
K.~Asano.
\newblock On local solutions of the initial value problem for the
  {V}lasov-{M}axwell equation.
\newblock {\em Comm. Math. Phys.}, 106(4):551--568, 1986.

\bibitem{AU}
K.~Asano and S.~Ukai.
\newblock On the {V}lasov-{P}oisson limit of the {V}lasov-{M}axwell equation.
\newblock In {\em Patterns and waves}, volume~18 of {\em Stud. Math. Appl.},
  pages 369--383. North-Holland, Amsterdam, 1986.

\bibitem{BGP}
F.~Bouchut, F.~Golse, and C.~Pallard.
\newblock Classical solutions and the {G}lassey-{S}trauss theorem for the 3{D}
  {V}lasov-{M}axwell system.
\newblock {\em Arch. Ration. Mech. Anal.}, 170(1):1--15, 2003.

\bibitem{B}
Y.~Brenier.
\newblock Convergence of the {V}lasov-{P}oisson system to the incompressible
  {E}uler equations.
\newblock {\em Comm. Partial Differential Equations}, 25(3-4):737--754, 2000.

\bibitem{BMP}
Yann Brenier, Norbert Mauser, and Marjolaine Puel.
\newblock Incompressible {E}uler and e-{MHD} as scaling limits of the
  {V}lasov-{M}axwell system.
\newblock {\em Commun. Math. Sci.}, 1(3):437--447, 2003.

\bibitem{DEG}
P.~Degond.
\newblock Local existence of solutions of the {V}lasov-{M}axwell equations and
  convergence to the {V}lasov-{P}oisson equations for infinite light velocity.
\newblock {\em Math. Methods Appl. Sci.}, 8(4):533--558, 1986.

\bibitem{Degond}
P.~Degond.
\newblock Spectral theory of the linearized {V}lasov-{P}oisson equation.
\newblock {\em Trans. Amer. Math. Soc.}, 294(2):435--453, 1986.

\bibitem{DPL}
R.~J. DiPerna and P.-L. Lions.
\newblock Global weak solutions of {V}lasov-{M}axwell systems.
\newblock {\em Comm. Pure Appl. Math.}, 42(6):729--757, 1989.

\bibitem{GSC97}
R.~Glassey and J.~Schaeffer.
\newblock The ``two and one-half-dimensional'' relativistic {V}lasov {M}axwell
  system.
\newblock {\em Comm. Math. Phys.}, 185(2):257--284, 1997.

\bibitem{GSC98}
R.~Glassey and J.~Schaeffer.
\newblock The relativistic {V}lasov-{M}axwell system in two space dimensions.
  {I}, {II}.
\newblock {\em Arch. Rational Mech. Anal.}, 141(4):331--354, 355--374, 1998.

\bibitem{GS86}
R.~Glassey and W.~Strauss.
\newblock Singularity formation in a collisionless plasma could occur only at
  high velocities.
\newblock {\em Arch. Rational Mech. Anal.}, 92(1):59--90, 1986.

\bibitem{Gr95}
E.~Grenier.
\newblock Defect measures of the {V}lasov-{P}oisson system in the quasineutral
  regime.
\newblock {\em Comm. Partial Differential Equations}, 20(7-8):1189--1215, 1995.

\bibitem{Gr96}
E.~Grenier.
\newblock Oscillations in quasineutral plasmas.
\newblock {\em Comm. Partial Differential Equations}, 21(3-4):363--394, 1996.

\bibitem{Gr99}
E.~Grenier.
\newblock Limite quasineutre en dimension 1.
\newblock In {\em Journ\'ees ``\'{E}quations aux {D}\'eriv\'ees {P}artielles''
  ({S}aint-{J}ean-de-{M}onts, 1999)}, pages Exp.\ No.\ II, 8. Univ. Nantes,
  Nantes, 1999.

\bibitem{G}
Emmanuel Grenier.
\newblock On the nonlinear instability of {E}uler and {P}randtl equations.
\newblock {\em Comm. Pure Appl. Math.}, 53(9):1067--1091, 2000.

\bibitem{G97}
Yan Guo.
\newblock Stable magnetic equilibria in collisionless plasmas.
\newblock {\em Comm. Pure Appl. Math.}, 50(9):891--933, 1997.

\bibitem{G99}
Yan Guo.
\newblock Stable magnetic equilibria in a symmetric collisionless plasma.
\newblock {\em Comm. Math. Phys.}, 200(1):211--247, 1999.

\bibitem{GS}
Yan Guo and Walter~A. Strauss.
\newblock Nonlinear instability of double-humped equilibria.
\newblock {\em Ann. Inst. H. Poincar\'e Anal. Non Lin\'eaire}, 12(3):339--352,
  1995.

\bibitem{GS99}
Yan Guo and Walter~A. Strauss.
\newblock Unstable oscillatory-tail waves in collisionless plasmas.
\newblock {\em SIAM J. Math. Anal.}, 30(5):1076--1114 (electronic), 1999.

\bibitem{GS00}
Yan Guo and Walter~A. Strauss.
\newblock Magnetically created instability in a collisionless plasma.
\newblock {\em J. Math. Pures Appl. (9)}, 79(10):975--1009, 2000.

\bibitem{HI}
D.~Han-Kwan and M.~Iacobelli.
\newblock {The quasineutral limit of the Vlasov-Poisson equation in Wasserstein
  metric}.
\newblock {\em Comm. Math. Sci., to appear}, 2014.

\bibitem{HI2}
D.~Han-Kwan and M.~Iacobelli.
\newblock {Quasineutral limit for Vlasov-Poisson via Wasserstein stability
  estimates in higher dimension}.
\newblock {\em Submitted}, 2015.

\bibitem{HKR}
D.~Han-Kwan and F.~Rousset.
\newblock {Quasineutral limit for Vlasov-Poisson with Penrose stable data}.
\newblock {\em Ann. Sci. \'Ecole Norm. Sup., to appear}, 2015.

\bibitem{HKH}
Daniel Han-Kwan and Maxime Hauray.
\newblock Stability {I}ssues in the {Q}uasineutral {L}imit of the
  {O}ne-{D}imensional {V}lasov--{P}oisson {E}quation.
\newblock {\em Comm. Math. Phys.}, 334(2):1101--1152, 2015.

\bibitem{KCI}
A.~Kingsep, K.~Chukbar, and V.~Ian'kov.
\newblock Electron magnetohydrodynamics.
\newblock {\em Voprosy Teorii Plazmy}, 16:209--250, 1987.

\bibitem{KS}
S.~Klainerman and G.~Staffilani.
\newblock A new approach to study the {V}lasov-{M}axwell system.
\newblock {\em Commun. Pure Appl. Anal.}, 1(1):103--125, 2002.

\bibitem{LS07b}
Zhiwu Lin and Walter Strauss.
\newblock Nonlinear stability and instability of relativistic
  {V}lasov-{M}axwell systems.
\newblock {\em Comm. Pure Appl. Math.}, 60(6):789--837, 2007.

\bibitem{LS07}
Zhiwu Lin and Walter~A. Strauss.
\newblock Linear stability and instability of relativistic {V}lasov-{M}axwell
  systems.
\newblock {\em Comm. Pure Appl. Math.}, 60(5):724--787, 2007.

\bibitem{LS08}
Zhiwu Lin and Walter~A. Strauss.
\newblock A sharp stability criterion for the {V}lasov-{M}axwell system.
\newblock {\em Invent. Math.}, 173(3):497--546, 2008.

\bibitem{LStr}
Jonathan Luk and Robert~M. Strain.
\newblock A new continuation criterion for the relativistic {V}lasov-{M}axwell
  system.
\newblock {\em Comm. Math. Phys.}, 331(3):1005--1027, 2014.

\bibitem{Mas}
N.~Masmoudi.
\newblock From {V}lasov-{P}oisson system to the incompressible {E}uler system.
\newblock {\em Comm. Partial Differential Equations}, 26(9-10):1913--1928,
  2001.

\bibitem{NStr}
Toan~T. Nguyen and Walter~A. Strauss.
\newblock Linear stability analysis of a hot plasma in a solid torus.
\newblock {\em Arch. Ration. Mech. Anal.}, 211(2):619--672, 2014.

\bibitem{Pazy}
A.~Pazy.
\newblock {\em Semigroups of linear operators and applications to partial
  differential equations}, volume~44 of {\em Applied Mathematical Sciences}.
\newblock Springer-Verlag, New York, 1983.

\bibitem{pen}
O.~{Penrose}.
\newblock {Electrostatic instability of a uniform non-Maxwellian plasma}.
\newblock {\em Phys. Fluids}, 3:258--265, 1960.

\bibitem{PSR}
Marjolaine Puel and Laure Saint-Raymond.
\newblock Quasineutral limit for the relativistic {V}lasov-{M}axwell system.
\newblock {\em Asymptot. Anal.}, 40(3-4):303--352, 2004.

\bibitem{SCH}
J.~Schaeffer.
\newblock The classical limit of the relativistic {V}lasov-{M}axwell system.
\newblock {\em Comm. Math. Phys.}, 104(3):403--421, 1986.

\bibitem{Wol1}
S.~Wollman.
\newblock An existence and uniqueness theorem for the {V}lasov-{M}axwell
  system.
\newblock {\em Comm. Pure Appl. Math.}, 37(4):457--462, 1984.

\bibitem{Wol2}
S.~Wollman.
\newblock Local existence and uniqueness theory of the {V}lasov-{M}axwell
  system.
\newblock {\em J. Math. Anal. Appl.}, 127(1):103--121, 1987.

\bibitem{Zum}
Kevin Zumbrun.
\newblock Planar stability criteria for viscous shock waves of systems with
  real viscosity.
\newblock In {\em Hyperbolic systems of balance laws}, volume 1911 of {\em
  Lecture Notes in Math.}, pages 229--326. Springer, Berlin, 2007.

\end{thebibliography}

\end{document}